\documentclass[10pt,reqno]{amsart}

\usepackage{epsfig}
\usepackage{enumerate}
\usepackage{amsfonts}
\usepackage{amssymb}
\usepackage{amsmath}   
\usepackage{amsthm}
\usepackage{latexsym}
\usepackage{amscd}
\usepackage[all]{xy}
\usepackage{verbatim}
\usepackage{color}
\usepackage{cite}

\newtheorem{theorem}{Theorem}[section]
\newtheorem{theoremletter}{Theorem}

\newtheorem{lemma}[theorem]{Lemma}
\newtheorem{corollary}[theorem]{Corollary}
\newtheorem{proposition}[theorem]{Proposition}
\newtheorem{definition}[theorem]{Definition}

\newtheorem{remark}[theorem]{Remark}
\newenvironment{eqn}{\begin{equation}}{\end{equation}}

\newtheorem{examples}[theorem]{Examples}

\newcommand{\K}{\operatorname{K}}
\def\e{\operatorname{e}}

\def\Hilb{\operatorname{Hilb}}

\def\chernchar{\operatorname{ch}}

\def\Td{\operatorname{Td}}

\def\Supp{\operatorname{Supp}}
\def\rk{\operatorname{rk}}
\def\Gr{\operatorname{Gr}}
\def\Ext{\operatorname{Ext}}

\def\Hom{\operatorname{Hom}}
\def\sign{\operatorname{sign}}
\def\Syst{\operatorname{Syst}}
\def\Pic{\operatorname{Pic}}

\newcommand{\into}{\rightarrow}

\newcommand{\cech}{\vee}
\newcommand{\chone}[1]{\mathrm{c}_1(#1)}
\newcommand{\ch}[2]{\mathrm{ch}_{#1}(#2)}
\newcommand{\chtwo}[1]{\ch{2}{#1}}

\newcommand{\cherntwo}[1]{\mathrm{c}_2(#1)}

\renewcommand{\phi}{\varphi}

\def\P{\mathbb{P}}
\def\Z{\mathbb{Z}}
\def\Q{\mathbb{Q}}

\def\C{\mathbb{C}}
\def\P{\mathbb{P}}

\def\O{\mathcal{O}}

\def\E{\mathcal{E}}
\def\F{\mathcal{F}}

\def\U{\mathcal{U}}

\newcommand{\SHom}{\mathcal{H}om}

\renewcommand{\bar}[1]{\overline{#1}}

\newcommand{\hodge}[1]{\mathrm{e}\left(#1\right)}
\newcommand{\sqnom}[2]{\genfrac{[}{]}{0pt}{}{#1}{#2}}
\newcommand{\curbnom}[2]{\displaystyle\genfrac{\{}{\}}{0pt}{}{#1}{#2}}
\newcommand{\Smat}[1]{\mathbf{Syst^{#1}(g)}}
\newcommand{\Mat}{\mathbf{M(g)}}
\newcommand{\Gmat}{\mathbf{M^0(g)}}
\newcommand{\Amat}[1]{\mathbf{A(#1)}}
\newcommand{\Pmat}[1]{\mathbf{P(#1)}}

\title{Higher rank stable pairs on K3 surfaces}
\date{\today}
\author{B. Bakker}
\address{
Courant Institute of Mathematical Sciences,
251 Mercer St., New York, NY 10012
}
\email{bakker@cims.nyu.edu}
\author{A. Jorza}
\address{
California Institute of Technology, Department of Mathematics, Pasadena, CA}
\email{ajorza@caltech.edu}

\begin{document}
\begin{abstract}
We define and compute higher rank analogs of Pandharipande-Thomas
stable pair invariants in primitive classes for K3 surfaces.  Higher rank stable pair
invariants for Calabi-Yau threefolds have been defined by Sheshmani
\cite{shesh1,shesh2} using moduli of pairs of the form $\O^n\into \F$ for $\F$ purely
one-dimensional and computed via wall-crossing techniques.  These invariants may be thought of as virtually
counting embedded curves decorated with a $(n-1)$-dimensional linear
system.  We treat invariants counting pairs $\O^n\into \E$ on a
$\K3$ surface for $\E$ an arbitrary stable sheaf of a fixed numerical
type (``coherent systems'' in the language of \cite{KY}) whose first Chern class is primitive, and fully compute them geometrically.  The ordinary stable pair theory of $\K3$ surfaces is treated
by \cite{MPT}; there they prove the KKV conjecture in primitive classes by showing the resulting partition functions are
governed by quasimodular forms.  We prove a ``higher'' KKV conjecture by showing that our higher
rank partition functions are modular forms.
\end{abstract}
\maketitle

\section{Introduction}
\subsection{Main results}
Stable pair invariants were defined for threefolds $X$ in \cite{PT1} by integration over a moduli space $P_k(X,\beta)$ parametrizing nonzero maps $\O_X\into\F$ with zero-dimensional cokernel for $\F$ a purely one-dimensional sheaf (see $\S \ref{threefold pairs}$) with $k=\chi(\F)$ and $[\Supp\F]=\beta$.  Recently these invariants have been generalized to counts of maps $\O_X^n\into \F$ for $n\geq 1$ and $\F$ higher rank (see $\S \ref{higher pairs}$).  The aim of this paper is to define and fully compute higher stable pair invariants for $X$ a K3 surface.  

Let $D$ be a divisor class on a K3 surface $X$ such that any representative of $D$ is reduced and irreducible (a divisor of minimal degree will be sufficient, \emph{cf.} Definition \ref{minimal degree}), $n,r$ nonnegative integers, and $k\in\Z$.  The Kawai-Yoshioka moduli space $\Syst_X^n(r,D,k)$ \cite{KY} of coherent systems parametrizes nonzero maps $\O_X^n\into \E$ with $\E$ stable of Mukai vector $v(\E)=(r,D,k)$.  It was originally noted by \cite{PT3} that with the above restriction on $D$, $\Syst^1(0,D,k)$ is isomorphic to $P_k(X,D)$ (which still exists for $X$ a surface, though the invariants are only defined in the threefold case).  Indeed, if $\P=|D|$ is the complete linear system of $D$ and
$X\times\P\supset\mathcal{C}_D\into \P$ is the universal divisor, then $\Syst^1(0,D,k)$ is simply the relative Hilbert scheme
$\mathcal{C}_D^{[k+g-1]}=\Hilb^{k+g-1}(\mathcal{C}_D/\P)$.  We therefore view $\Syst_X^n(r,D,k)$ for $n>1$ or $ r>0$ as a moduli space of higher stable pairs.   $\Syst_X^n(r,D,k)$ is smooth \cite[Lemma 5.117]{KY}, so we define the signed Euler characteristic of $\Syst_X^n(r,D,k)$ to be the higher stable pair invariant, in analogy with the threefold case.  The Euler characteristic is deformation invariant for deformations of $X$ for which $D$ remains algebraic and such that every representative is reduced and irreducible, so for each genus $g$ we once and for all fix a K3 surface $X_g$ with such a divisor $D_g$ of genus $g$ (see $\S \ref{preparations}$) and compute it's higher stable pair invariants.  

Our main result is a computation of the Hodge polynomials
$\hodge{\cdot}=\sum h^{p,q}(\cdot)(-t)^p(-\bar{t})^q$ of the moduli spaces $\Syst_{X_g}^n(r,D_g,k)$:  defining generating functions 
\begin{eqn}\label{gen}F^r_n(q,y)=\sum_{g\geq 0}\sum_{k\in\mathbb Z}\hodge{\Syst_{X_g}^n(r,D_g,k+r)}(t\bar t)^{-g}y^kq^{g-1}\end{eqn}
we prove in Theorem \ref{thm}:
\begin{theoremletter}\label{maintheorem} Let $S(q)=\displaystyle\sum_{n\geq 0}\e(X^{[n]})(t\bar t)^{-n}q^{n-1}$ be the generating function of the (symmetrized) Hodge polynomials of the Hilbert schemes $X^{[n]}$ of $n$ points on a K3 surface $X$.  For $X_g,D_g$ chosen as above, 
\[\frac{F^r_n(q,y)}{S(q)}=\frac{(t\bar t)^{r(n-r)}}{[n]}\sum_{\substack{p\geq n-r\\\ell\geq r}}(t\bar t)^{-n\ell-(p-\ell)r}[p+\ell]\sqnom{n+\ell-r-1}{n-1}\sqnom{p+r-1}{n-1}y^{p-\ell}q^{p\ell}\]
\end{theoremletter}

The square binomial coefficient $\sqnom{n}{k}$ is a polynomial
in $u=t \overline{t}$ (see \S \ref{uint}) which computes the
Hodge polynomial of $\Gr(k,n)$ while $[n]$ is the Hodge
polynomial of $\mathbb{P}^{n-1}$.
%Setting $n=1,r=0$ we rederive the calculation of \cite{KY}, and in particular the potential $Z_g(y)$:
%\begin{corollary}
%\[F(q,y)=\left. F^0_1(q,y)\right|_{t=\bar t=1}\]
%\end{corollary}
%This is equation (4) of \cite{MPT}.
The technique involved in the proof is a generalization of the calculation of \cite{KY}, and we reproduce their result for $F_0^1(q,y)$. We remark that this rank 1 generating function $F_0^1(q,y)$ is related \cite{MPT} to the reduced Gromov-Witten potentials of the $\K3$ surface via a change of variables (for details see \S \ref{k3 pairs}); although there is currently no notion of ``higher Gromov-Witten theory," we expect there to be wall-crossing relationships between our invariants and other higher rank analogues of ``sheaf-theoretic" curve-counting invariants on K3 surfaces.  

% and we deduce a formula for the Hodge polynomials of the
%Brill-Noether strata of all moduli spaces of sheaves with primitive
%first Chern class on a K3 surface. WHERE??????? Remark 3.6 seems incomplete.

Using Theorem \ref{maintheorem} we further show that the higher partition functions $F^r_n(q,y)$ are governed by modular forms (Theorem \ref{thm:modularity}):
\begin{theoremletter}\label{maintheorem modularity}  Substituting $y=e^{iv}$, the coefficient of $v^s$ in the Taylor series expansion of $(F_n^r(q,y)/S(q))|_{t=\overline{t}=1}$ is an element of a $\mathbb{Q}(i)$-algebra generated by Eisenstein series of level $\Gamma(4)$ (cf. \S \ref{euler}).
\end{theoremletter}

The proof of Theorem \ref{maintheorem modularity} relies on Hickerson's work on Ramanujan's mock
theta conjectures (\emph{cf}. Theorem \ref{modus}); the mock theta conjectures state that certain
mock theta functions (which Ramanujan defined as certain generating
functions, but can be thought of as the holomorphic parts of
certain Maass forms) can be written as linear combinations of
infinite products.

% We view Theorem \ref{maintheorem modularity} as a Katz-Klemm-Vafa (KKV) conjecture for higher stable pair invariants.  The conjecture predicts on the basis of a string theoretic duality that curve-counting invariants of K3 surfaces should in general assemble into modular forms (see $\S\ref{k3 pairs}$ for further discussion).
Theorem \ref{maintheorem modularity}, generalizing the $r=0, n=1$ result of \cite{MPT} (see \S \ref{k3 pairs} (c)), is
surprising in that it is not predicted by physics.  The modularity of
the ordinary stable pair and Gromov-Witten generating functions of a
K3 surface are physically attributed by Katz, Klemm, and Vafa
\cite{KKV} to the duality between M-theory compactified on a K3
surface and heterotic string theory compactified on $T^3$ (here
$T=S^1$) \cite{heterotic}.  We will hereafter refer to mathematical
statements of the modularity of such generating functions as the KKV
conjecture; it has been proven in several cases.  The relative Hilbert
scheme $\mathcal{C}^{[d]}_D$ is interpreted by \cite{KY} as a space of
$D0$-branes bound to a $D2$-brane wrapping a K3 surface, and their
calculation \eqref{ky compute} proves the KKV conjecture for such
invariants, for $D$ of minimal degree (\emph{cf.} Definition
\ref{minimal degree}).  \cite{MPT} proves the KKV conjecture for
Gromov-Witten potentials in primitive classes, which by an MNOP-style
duality (see (b) of \S1.4 below) agrees with our Theorem \ref{maintheorem modularity} for $r=0$ and $n=1$, and \cite{localk3} conjecturally treats the KKV conjecture for ordinary stable pair invariants in all divisor classes.  In its full generality, Theorem \ref{maintheorem modularity} should be interpreted as a higher rank KKV conjecture in classes of minimal degree, though it is an interesting question whether our generating functions have a physical interpretation.

To further motivate our results in the remainder of the introduction we review stable pair theories for threefolds and $\K3$ surfaces.

\subsection{Stable pair invariants on threefolds}\label{threefold pairs}
Let $X$ be a smooth threefold.  A stable
pair is a one-dimensional sheaf $\F$ together with a nonzero section
$\O_X\into \F$ whose kernel is zero-dimensional.  The moduli space of stable pairs with $[\Supp(\F)]=\beta$ and $\chi(\F)=k$ is a projective scheme $P_k(X,\beta)$ (see \cite{PT1} for
details).  Generically, the support $C=\Supp(\F)$ of $\F$ is a smoothly embedded
curve, in which case $\F$ is a line bundle $L_C$ on $C$ and the section
$\O_X\into \F$ is a composition $\O_X\into\O_C\into
L_C$, where the latter map is a section of $L_C$, i.e., a
divisor on $C$ in the divisor class given by $L_C$.  Thus,
$P_k(X,\beta)$ is a compactification of the space of smoothly
embedded, $(k+g-1)$-pointed curves.

Recall that the Behrend function
$\nu_M:M\into \Z$ of a
scheme $M$ is a canonical constructible function associated to $M$
which measures the singularities of $M$ (see \cite{chibehrend} for basic
properties); for example, if $M$ is smooth then $\nu_M$ is constant,
equal to $(-1)^{\dim M}$.  By
integrating $\nu_M$ with respect to the (topological) Euler
characteristic measure $d\chi$ on $M$ we obtain an invariant.  For $M=P_k(X,\beta)$, Behrend has shown \cite{chibehrend} that
this yields the Pandharipande-Thomas stable pair invariants of $X$
\[PT_{\beta,k}=\int_{P_k(X,\beta)}\nu_{P_k(X,\beta)}d\chi:=\sum_{s\in\Z}s\chi(\nu_{P_k(X,\beta)}^{-1}(s))\]
which can be thought of as a virtual count of pointed curves.  This number was originally defined by integrating the
virtual class of a symmetric obstruction theory on $P_k(X,\beta)$
in \cite{PT1}, and using those techniques can be shown to be deformation invariant of
$(X,\beta)$.  Note that if $P_k(X,\beta)$ is smooth, then
$PT_{\beta,n}$ is the signed Euler characteristic of $P_k(X,\beta)$.

The Donaldson-Thomas invariant $DT_{\beta,k}$ of $X$ (for $X$ Calabi-Yau) is defined similarly by integrating the virtual class of a symmetric obstruction theory on $I_k(\beta,X)$, the Hilbert scheme of subschemes $Z\subset X$ such
that $[Z]=\beta$ and $\chi(\O_Z)=k$, and is once again a deformation invariant \cite{dt}.  $DT_{\beta,k}$ can likewise be shown to be equal to the integral of the Behrend function of $I_k(\beta,X)$.  
Both $I_k(X,\beta)$ and $P_k(X,\beta)$ can be thought of as
parametrizing pairs $\O\xrightarrow{s}\F$ with $\F$ one-dimensional, though
with respect to different stability conditions:  in Donaldson-Thomas
theory, we require $s$ to be surjective; in Pandharipande-Thomas
theory, $\F$ is required to be pure and $s$ has zero-dimensional kernel.  It is
therefore not surprising that Donaldson-Thomas and Pandharipande-Thomas invariants are related to
each other via a wall-crossing formula (there is a great deal of literature on this---see \emph{e.g.} \cite{bridgeland2011hall,Joyce,KS,toda}).
\subsection{Higher rank stable pair theories}\label{higher pairs}
There are two means by which one can generalize either
of the above invariants on threefolds to higher rank:  if $I_k(X,\beta), P_k(X,\beta)$
parametrize pairs $\O_X\into\F$, the higher moduli spaces
$I_k^{r,n}(X,\beta),P_k^{r,n}(X,\beta)$ should parametrize pairs
$\O_X^n\into \F$ with $\rk \F=r$.  We refer to $n$
as the section rank and $r$ as the sheaf rank.

Both generalizations have been partially treated in recent literature
for $X$ a Calabi-Yau threefold:
\subsubsection*{Higher section rank}  A clear
candidate for $I_k^{0,n}(X,\beta)$ is the Quot scheme of 1-dimensional quotients $Q$ of $\O_X^n$ with
$\chi(Q)=k$ and $[\Supp Q]=\beta$.  These invariants\footnote{The moduli space considered by Toda is not exactly the Quot scheme; there is an additional stability condition.} are computed for
$\beta=0$ and $n=2$ by Toda in \cite{todahigher}, and for more general section
rank $n$ by \cite{nagaohigher} by relating the resulting moduli spaces to quiver
varieties.  Both computations rely on Joyce's wall-crossing formulae \cite{Joyce,KS}.
\subsubsection*{Higher sheaf rank}  The higher sheaf rank moduli spaces
in our sense have not been considered, though for Calabi-Yau
threefolds $X$ Sheshmani \cite{shesh1,shesh2} has defined and computed invariants counting
stable pairs of the form $\O_X(-\ell)^n\into F$ for $F$ pure and
1-dimensional, but with arbitrary rank \emph{on its support}---that
is, $\cherntwo{F}=r[\Supp(F)]$.  Once again, his computations rely on
Joyce's wall-crossing machinery and virtual localization.
\vskip 1em
We view the moduli space $\Syst^n(r,D,k)$ as simultaneously achieving the two analogous generalizations to both higher section rank $n$ and higher sheaf rank $r$ in the surface case.
\subsection{Previous work on stable pairs on $\K3$ surfaces}\label{k3 pairs}  
Let $X$ be a K3 surface and $D$ a divisor class such that every
divisor in $D$ is reduced and irreducible of genus $g$ (again, more generally for $D$ of minimal agree \emph{cf.} Definition \ref{minimal degree}).  Following
\cite{KY}, let $\P=|D|$ be the complete linear system of $D$ and
$X\times\P\supset\mathcal{C}_D\into \P$ the universal divisor.  As noted by \cite{PT3}, the relative Hilbert scheme
$\mathcal{C}_D^{[d]}=\Hilb^d(\mathcal{C}_D/\P)$ parametrizing divisors
$C$ in the class $D$ and subschemes $Z$ of $C$ of length $d$ is the moduli space $P_k(X,D)$ of stable pairs
$\O_X\into \F $ with $\chone{\F}=D$ and $\chi(\F)=d+1-g=k$, where $D^2=2g-2$.
$\mathcal{C}^{[d]}_D$ is smooth, so a reasonable replacement for the
Pandharipande-Thomas invariant is the (signed) topological Euler
characteristic of $\mathcal{C}^{[d]}_D$ \cite{PT3}:
\[N_{D,d}=(-1)^{d+g}\chi(\mathcal{C}^{[d]}_D)\]
Indeed, \cite[\S 3.7]{MPT} show that this invariant can be directly computed from the threefold theory; it is the same as the invariant associated with the virtual class obtained by restricting the symmetric obstruction theory on $P_k(X\times \C,i_* D)$ to $P_k(X,D)  \lhook\joinrel\xrightarrow{i} P_k(X\times \C,i_* D)$ after embedding $i:X\into X\times \C$ as the fiber over $0\in \C$.\footnote{Starting from this construction, Kool and Thomas \cite{KT1,KT2} have more recently defined stable pair invariants for a wider class of surfaces $X$ as an equivariant residue of the threefold invariants of $X\times\C$.  The resulting obstruction theory on $P_k(X,D)$ is not of virtual dimension 0, and the invariants with insertions are therein related to G\"ottsche invariants \cite{Got} (see also \cite{KST} where this is used to prove the G\"ottsche conjecture).  It would be interesting to see if higher rank analogs of these invariants can be defined.}
These invariants are typically organized into generating functions
\[Z_D(y)=\sum_{k\geq 1-g}(-1)^{k+2g-1}\chi(\mathcal{C}_D^{[k+g-1]})y^{k}\]
The functions $Z_D(y)$ are studied in detail by \cite{MPT} for primitive
$D$.  There they
show:
\begin{enumerate}
\item[(a)] $Z_D(y)=Z_g(y)$ only depends on the genus $g$ of $D$, and by \cite[Theorem 5.80]{KY}
\begin{align}F(q,y):&=-\sum_{g\geq 0}q^{g-1}Z_g(-y)\notag\\
&=\frac{1}{\left(\sqrt{y}-\frac{1}{\sqrt{y}}\right)^{2}}\prod_{n\geq 1}\frac{1}{(1-yq^n)^2(1-q^n)^{20}(1-y^{-1}q^n)^2}\notag\\
&=\frac{s(q)}{\left(\sqrt{y}-\frac{1}{\sqrt{y}}\right)^{2}}\prod_{n\geq 1}\frac{(1-q^n)^4}{(1-yq^n)^2(1-y^{-1}q^n)^2}\label{ky compute}
\end{align}
where $s(q)=S(q)|_{t=\bar t=1}=\prod_{n\geq 1}(1-q^n)^{-24}$ is the generating function of the Euler characteristics of the Hilbert schemes of points on a K3.
\item[(b)] $Z_g(y)$ is related to the reduced Gromov-Witten potentials of a
  $\K3$ surface via a change of variables $-y=e^{iv}$ analogous to the duality for
  Calabi-Yau threefolds conjectured in \cite{mnop1,mnop2}.
\item[(c)] The coefficient of $v^i$ in the full partition function 
\[F(q,y)=\sum_{g\geq 0}Z_g(y)q^{g-1}\]
after the substitution $y=e^{iv}$ is the $q$-expansion of a
quasi-modular form \cite[Theorem 4]{MPT}.
\end{enumerate}

\subsection*{Outline}  The outline of the paper is as follows.  In Section \ref{sect:2} we recall the moduli theory of stable pairs on a K3 surface $X$.  The key relationship between the relevant moduli spaces is developed in \S\ref{properties}.  In Section \ref{sect:3} we compute the generating functions \eqref{gen} using the geometry from Section \ref{sect:2}.  In \S\ref{subsect:KY} we express the general invariants in terms of the $r=0,n=1$ theory; in \S \ref{euler}, we compute the generating functions of the Euler characteristics and prove that the $v$-coefficients, after setting $y=e^{iv}$, are modular forms.  The less enlightening computations used in the course of Section \ref{sect:3} are collected in Section \ref{sect:4}.

\subsubsection*{Acknowledgements}The authors would like to thank D. Maulik and J. Tsimerman for valuable conversations.  The first author would like to thank R. Pandharipande in particular for introducing the authors to the subject matter and for many enlightening discussions which greatly improved the content and exposition of the paper.  The second author is grateful for Dinakar Ramakrishnan for helpful comments. Both authors would finally like to thank the referee for many helpful suggestions.  Part of the research reported here was completed while the authors were graduate students at Princeton University. 

\section{The moduli theory of sheaves and stable pairs on K3 surfaces}\label{sect:2}
Throughout this section, let $X$ be an algebraic K3 surface over $\C$.  The Mukai lattice of $X$ is the total cohomology ring $H^*(X,\Z)$ together with the pairing
\[(v,w)=-\int_{X}v^\cech w=\int_X(v_1w_1-v_0w_2-v_2w_0)\]
where for $v=v_0+v_1+v_2\in H^*(X,\Z)$, $v_i\in H^{2i}(X,\Z)$ are the homogeneous components, and similarly for $w$.  We will denote by $\omega\in H^4(X,\Z)$ the Poincar\'{e} dual to the point class.  Using the canonical isomorphisms $H^0(X,\Z)\cong\Z$ and $H^4(X,\Z)\cong\Z$, we will write $v=(r,D,a)$ for integers $r,a$ when $v_0=r$, $v_1=D$, $v_2=a\omega$.  Note that
\[\Td(X)=1+2\omega\]
Given a coherent sheaf $\E$ on $X$, the Mukai vector of $\E$ is
\begin{align*}
v(\E)&=\chernchar{\E}\sqrt{\Td(X)}\\
&=\rk(\E)+\chone{\E}+\left(\rk(\E)\omega+\frac{\chone{\E}^2}{2}-\cherntwo{\E}\right)\\
&=(\rk(\E),\chone{\E},\chi(\E)-\rk(\E))
\end{align*}
by Gronthendieck-Riemann-Roch.  The Mukai pairing is defined so that, for any coherent sheaves $\E,\F$ on $X$,
\[(v(\E),v(\F))=-\chi(R\Hom(\E,\F))\]
Most of the following sections are adapted from the treatment in \cite{KY}.
\subsection{Moduli of Sheaves}\label{moduli of sheaves}
Let $H$ be an ample divisor on $X$, $v=(r,D,a)\in H^*(X,\Z)$ a Mukai vector, and assume $v_1=D$ is primitive.  Recall that a coherent sheaf $\E$ on $X$ is Gieseker stable (resp. semistable) if for any subsheaf $\F\subset \E$, the Hilbert polynomials satisfy $\chi(\F\otimes H^n)<\chi(\E\otimes H^n)$ (resp. $\chi(\F\otimes H^n)\leq\chi(\E\otimes H^n)$) for $n\>>0$.  Throughout the following, by (semi)stability we will mean Gieseker (semi)stability with respect to $H$.  Let $M(v)$ be the moduli space of semistable sheaves $\E$ with $v(\E)=v$.  A well known theorem of Huybrechts \cite{huydef} (for a nice exposition see \cite[6.2.16]{huybrechts}) states that
\begin{theorem}\label{generic}  For generic $H$, $M(v)$ is a smooth projective irreducible symplectic variety of dimension $2+(v,v)=2(g-ra)$ deformation equivalent to the Hilbert scheme of $g-ra$ points $X^{[g-ra]}$ on $X$.
\end{theorem}

We will be concerned with the case when $D$ is of minimal degree:
\begin{definition}\label{minimal degree} A divisor class $D\in\Pic(X)$ has minimal degree if $D.H>0$ and no positive line bundle has smaller intersection product with $H$, that is
\[D. H=\min\{L.H|L\in\Pic(X), L.H>0\}\]
\end{definition}
Clearly every divisor of minimal degree is primitive.  The main importance of this definition is that for any divisor class $D$ of minimal degree, every divisor in that class is integral, and therefore moduli spaces of sheaves $\E$ with $v_1=D$ will be well-behaved.  For any genus $g$, there is a suitable K3 surface with a divisor class $D$ of genus $g$ and minimal degree:

\begin{examples}\label{examples}
\begin{enumerate}
\item If $X$ is an elliptic K3 surface with section, $\Pic(X)=\Z\sigma\oplus\Z f$, where $f$ is the fiber class and $\sigma$ the section class.  Choosing $H=\sigma+3f$ to be the ample class, we have
\[(a\sigma+b f).H=a+b\]
$\sigma$ and $f$ are clearly of minimal degree, since both have intersection product 1 with $H$.
\item If $X$ has Picard rank one and $H$ is the ample generator, then $D=H$ has minimal degree.
\end{enumerate}
\end{examples}
\begin{lemma}If $v_1=D$ is of minimal degree with respect to $H$, then
  $H$ is generic in the sense of Theorem \ref{generic}.
\end{lemma}
\begin{proof}This follows from the fact the semistability implies stability.
When $r>0$,  for sheaves $\E$ with $v(\E)=(r,D,a)$ of minimal degree, if $\F\subsetneq \E$ with $\frac{\chone{\F}.H}{\rk \F}<\frac{\chone{\E}.H}{\rk\E}$ then clearly $\chi(\F\otimes H^n)<\chi(\E\otimes H^n)$ for $n>>0$ since the equality either holds for the leading coefficient ($\frac{\rk\F}{2}<\frac{\rk \E}{2}$) or for the next coefficient.  If $\frac{\chone{\F}.H}{\rk \F}=\frac{\chone{\E}.H}{\rk\E}$, then $\rk\F=\rk\E$ and $\chone{\F}.H=\chone{E}.H$, for otherwise $\det\F$ would be a positive line bundle with smaller degree.  $\E/\F$ then has dimension 0 and again $\chi(\F\otimes H^n)<\chi(\E\otimes H^n)$ for $n>>0$.

When $r=0$, semistability and stability are both equivalent to purity even without the assumption of minimal degree.
%\color{red}BETTER CITATION.
\end{proof}
\subsection{A Stratification of the Moduli Spaces}\label{stratification}
In the setup of $\S\ref{moduli of sheaves}$ suppose further that $M(v)$ is a fine moduli space, so there exists a universal sheaf $\F$ on $X\times M(v)$, flat over $M(v)$, such that for every point $p=[\E]\in M(v)$, the restriction of $\F$ to $X\times p$ is $\E$.  For our purposes we need only consider the case when the Euler characteristic $\chi=-(v(\O),v)\geq0$ (\textit{cf.}, Remark \ref{restrict}).  Let $\pi:X\times M(v)\into M(v)$ be the projection, and  consider the subsets, 
\begin{eqn}\label{strat}
M(v)_i=\{[\E]\in M(v)|\dim H^0(\E)=i\}
\end{eqn}with the induced reduced subscheme structure.  By the semicontinuity theorem, we have immediately 
\begin{lemma}If $M(v)$ is a fine moduli space, then $\{M(v)_i\}_{i\geq 0}$ is a locally closed stratification of $M(v)$.
\end{lemma}
In general $M(v)$ need not have a universal family, but \'etale-locally it does.  The cohomology of coherent sheaves can be computed \'etale-locally, and closed and open immersions are both \'etale local properties, so
\begin{proposition}$\{M(v)_i\}_{i\geq 0}$ is a (finite) locally closed stratification of $M(v)$.
\end{proposition}%The finiteness follows from the coherence of $\pi_*\F$ \'etale-locally.  
\begin{remark}\label{generic piece}
Since the second cohomology vanishes for any stable sheaf $\E$ with Mukai vector $v$, 
\[\dim H^0(\E)\geq \chi(\E)=\chi=r+a\]
From Brill-Noether theory, we know for $D$ of minimal degree that:  (i) the generic stratum is in fact $M(v)_{r+a}$; (ii) each $M(v)_i$ for $0\leq i<r+a$ is empty; and (iii) each $M(v)_i$ for $i\geq r+a$ is of the expected dimension (when the expected dimension is nonnegative).  See for example \cite{BNtheory}.
\end{remark}
\subsection{Properties of Stable Pairs on K3 Surfaces}\label{properties}
Throughout this section, (semi)stability will mean Gieseker (semi)stability.

We briefly recall Le Potier's notion of a coherent system \cite{LP1}, henceforth referred to as a stable pair\footnote{There are many overlapping notions of stability of maps of sheaves, and equally many disparate terminologies (\emph{e.g.} coherent systems, framed sheaves, stable pairs).  The objects we study are closest to coherent systems, but we refer to them throughout as ``stable pairs" for general ease of exposition and because we are ultimately interested in interpreting coherent systems as higher rank stable pairs, in primitive divisor classes.  In this case coherent systems of dimension 1 with $\dim U=1$ and stable pairs in the sense of \cite{PT1} in fact coincide; in non-primitive divisor classes they do not.  We thank the referee for bringing this to our attention.}
\begin{definition}A stable pair $(U,\E)$ on $X$ is a stable sheaf $\E$ and a subspace $U\subset \Hom(\O,\E)$.  We will often denote a stable pair $(U,\E)$ by the corresponding evaluation map $U\otimes\O\into\E$.  A morphism $(U,\E)\into(U',\E')$ consists of morphisms $U\into U'$ and $\E\into\E'$ such that
\[\xymatrix{
U\otimes\O\ar[r]\ar[d]&\E\ar[d]\\
U'\otimes\O\ar[r]&\E'\\
}\]
commutes.  The Mukai vector of a stable pair $(U,\E)$ is the Mukai vector of $\E$, and the section rank of $(U,\E)$ is $\dim U$.
\end{definition}
There is an obvious relative notion of stable pair.  For a scheme $S$, let $\pi:X\times S\into S$ be the projection.  A family of stable pairs $(\mathcal{U},\E)$ on $X\times S/S$ is a sheaf $\E$ on $X\times S$ flat over $S$, a locally free sheaf $\mathcal{U}$ on $S$, and a morphism $\pi^*\mathcal{U}\into \E$ such that the restriction to each fiber of $\pi$ is a stable pair in the usual sense.  A morphism of relative stable pairs $(\mathcal{U},\E)$ and $(\mathcal{U'},\E')$ is again given by morphisms $\mathcal{U}\into \mathcal{U}'$ and $\E\into\E'$ such that
\[\xymatrix{
f^*\U\ar[r]\ar[d]&\E\ar[d]\\
f^*\U'\ar[r]&\E'\\
}\]
commutes.  By \cite{LP1} the moduli functor of stable pairs with Mukai vector $v$ and section rank $n$ is (coarsely) representable by a projective scheme $\Syst^n(v)$, and the obvious forgetful morphism $p:\Syst^n(v)\into M(v)$ is projective.

The following lemma of Yoshioka will control the geometry of $\Syst^n(v)$:

\begin{lemma}\label{result}Let $X$ be a K3 surface, $D$ a divisor on $X$ of minimal degree, and $\E$ a stable sheaf on $X$ with $\chone{\E}=D$.  Then
\begin{enumerate}
\item Given a subspace $U\subset \Hom(\O,\E)$, let $\phi:U\otimes\O\into\E$ be the evaluation map.  Either
\begin{enumerate}
\item $\dim U<\rk \E$, in which case $\phi$ is injective,
\[0\into U\otimes\O\into\E\into\F\into 0\]
and the quotient $\F$ is stable.
\item $\dim U\geq \rk \E$, in which case $\phi$ is not injective,
\[0\into\F\into U\otimes\O\into\E\into Q\into 0\]
and the kernel is stable and locally free, while the quotient $Q$ is dimension 0.
\end{enumerate}
\item Given $V\subset \Ext^1(\E,\O)$, then in the corresponding extension
\[0\into V^*\otimes\O\into \F\into\E\into0 \]
$\F$ is stable.
\end{enumerate}
\end{lemma}
\begin{proof}
See \cite[Lemma 2.1]{reflect}.
\end{proof}
This has a number of geometric consequences.  For example, we have 
\begin{theorem}\label{main}\cite[Lemma 5.117]{KY}.  Let $X$ be a K3 surface, $v\in H^*(X,\Z)$ a Mukai vector.  For $v_1=D$ of minimal degree, $\Syst^n(v)$ is smooth.
\end{theorem}
Whenever $-(v(\O),v)\geq 0$, denote by $\Syst^n(v)_i$ the preimage of the stratum $M(v)_i$ from $\S\ref{stratification}$ under the forgetful morphism $p:\Syst^n(v)\into M(v)$; clearly $\{\Syst^n(v)\}_{i\geq 0}$ is a locally closed stratification of $\Syst^n(v)$.

For $v=(r,D,a)$, denote $\Syst^n(r,D,a)=\Syst^n(v)$ and
$M(r,D,a)=M(v)$.  For $r\geq n$ there is a map (\emph{cf.}
\cite{KY}) \[q:\Syst^n(r,D,a)\into M(r-n,D,a-n)\] mapping $(\E,U)$ to
the cokernel $\F$ of the evaluation map $U\otimes\O\into\E$, which is
injective by Lemma \ref{result}: 
\[0\into U\otimes \O\into \E\into\F\into 0\]
Again by Lemma \ref{result} $\F$ is stable, and $v(\F)=v(\E)-v(\O^n)=(r-n,D,a-n)$ since $v(\O)=(1,0,1)$.  Further, since $H^1(U\otimes \O)=0$, the stratum $\Syst^n(r,D,a)_i$ maps into $M(r-n,D,a-n)_{i-n}$, assuming $r+a-2n\geq 0$.
\begin{lemma}\cite[Lemma 5.113]{KY}.  For $-(v(\O),v)\geq 0$,\begin{enumerate}
\item The restriction $\Syst^n(v)_i\into M(v)_i$ of the forgetful morphism $p$ is an \'{e}tale-locally trivial fibration with fiber $\Gr(n,i)$.
\item Furthermore, if $r+a\geq 2n$, then the restriction $\Syst^n(r,D,a)_i\into M(r-n,D,a-n)_{i-n}$ of the quotient morphism $q$ is an \'{e}tale-locally trivial fibration with fiber $\Gr(n,n+i-r-a)$.
\end{enumerate}
\end{lemma}
\begin{proof}Both parts are obvious if $M(v)$ has a universal sheaf $\F$, in which case $\Syst^n(v)$ is a relative Grassmannian of $\F$.  A universal sheaf exists \'{e}tale locally, and the result follows.  See \cite{KY}.
\end{proof}
% \begin{comment}
% \begin{remark}
% Note this allows us to easily compute the expected dimension of $\Syst^n(r,D,a)$ from the dimension of $M(r,D,a)$, for $n\leq r+a$:
% \[\dim\Syst^n(r,D,a)=2(g-ra)+n(a+r-n)=2g+a(n-2r)+n(r-n)\]
% Also, fixing $n,r,D$, only finitely many of the $\Syst^n(r,D,a)$ are nonempty; 
% \end{remark}
% \end{comment}
The main tool for the computation of the Hodge polynomials of $\Syst^n(r,D,a)$ will be the existence of the resulting diagrams
\[\xymatrix{
&{\Syst^n(r,D,a)_i}\ar[dr]^q\ar[dl]_p&\\
{M(r,D,a)_i}&&{M(r-n,D,a-n)_{i-n}}
}\]
where $p$ is an \'etale-local $\Gr(n,i)$-fibration and $q$ is an \'etale local $\Gr(n,n+i-r-a)$-fibration.

One final property of the stable pair moduli spaces that will be relevant later is the duality, first proven by \cite[Theorem 39]{Markman}
\begin{proposition}\cite[Proposition 5.128]{KY}.\label{dual}  In the
  setup of Theorem \ref{main} there is an isomorphism
\[\Syst^n(r,D,a)\cong\Syst^n(n-r,D,a-r)\]
for all $r\leq n$.
\end{proposition}
\begin{remark}\label{restrict}By this duality, if we're interested in
  $\Syst^n(r,D,k+r)$ for $r\leq n$, we may assume $k\geq 0$; indeed,
  the duality is equivalent to
  $\Syst^n(r,D,a)\cong\Syst^n(n-r,D,a+(n-r))$.  Thus we need only consider moduli spaces involving sheaves of nonnegative Euler characteristic.
\end{remark}
\begin{proof}[Proof of Proposition \ref{dual}]  We will at the very least define the map; see \cite{KY} for a proof of the theorem.  Let $U\otimes\O\into\E$ be a stable pair, and consider $U\otimes\O\into\E$ as a morphism of complexes supported in degree 0 in the derived category $D^b(X)$; let $x\in D^b(X)$ be the cone.  Thus, there is a triangle
\begin{equation}\label{triangle}x\into U\otimes\O\into \E\into x[1]\end{equation}
Alternatively, we can think of $x$ as the 2-term complex $[U\otimes \O\into \E]$ with $\E$ placed in degree 1.  Applying $R\SHom(\,\cdot\,,\O)$ to the triangle \eqref{triangle}, we have a morphism 
\begin{eqn}\label{dualness}U^*\otimes\O\cong\SHom(U\otimes\O,\O)\into \SHom(x,\O)\end{eqn}One can show that $U^*\otimes\O\into\SHom(x,\O)$ is a stable pair and that this defines the isomorphism.  For example, \eqref{dualness} is injective on global sections because, applying $R\Hom(\,\cdot\,,\O)$ to \eqref{triangle}, there is an exact sequence
\begin{equation}\label{othertriangle}0\cong\Hom(\E,\O)\into U^*\into \Hom(x,\O)\into \Ext^1(\E,\O)\into 0\end{equation}
where the triviality of $\Hom(\E,\O)$ follows from the stability of $\E$.
\end{proof}
\begin{remark}\label{stratumrmk}
In fact, by \eqref{othertriangle}, we obtain an isomorphism
\[\Syst^n(r,D,a)_i\cong\Syst^n(n-r,D,a-r)_{i+n-\chi}\]
where $\chi=-(v(\O),v), v=(r,D,a)$.
\end{remark}

\section{Computation of hodge polynomials}\label{sect:3}
This section will be devoted to computing the generating functions of the moduli spaces of stable pairs on K3 surfaces.  The geometric arguments are given here; some useful computations are collected in the subsequent section.
\subsection{Preparations}\label{preparations}
For $X$ a scheme over $\C$, let 
\[\hodge{X}=\sum_{p,q\geq 0}h^{p,q}(X)(-t)^p(-\bar{t})^q\] 
denote the virtual Hodge polynomial (Hodge-Deligne polynomial) of $X$ \cite{deligne}.  Throughout the following, we will set $u=t\bar{t}$; the Hodge polynomial of the Grassmannian $\Gr(k,n)$ of $k$ planes in $n$-space is easily expressed in terms of $u$-integers (see $\S\ref{uint}$): 
\[\hodge{\Gr(k,n)}=\sqnom{n}{k}\]
In particular,
\[\hodge{\P^n}=[n+1]\]
Let $X$ now be a K3 surface.  Recall that for a divisor class $D\in
H^2(X,\Z)$, $D^2=2g-2$ by the adjunction formula, where $g$ is the
arithmetic genus of a divisor in the class $D$; $g$ will be called the
genus of $D$.  For each genus $g\geq 0$ fix a polarized K3 surface
$X_g$ with polarization $H_g$ and a divisor class $D_g$ of minimal
degree and genus $g$, \emph{cf.} Examples \ref{examples}: 
\begin{itemize}
\item $g=0,1$:  $X_g\into \P^1$ is an elliptic K3 with a section.  $\Pic(X_g)=\Z\sigma\oplus\Z f$, where $f$ is the fiber class and $\sigma$ the section class.  For $g=0$ take $H_0=\sigma+3 f$ and $D_0=\sigma$; for $g=1$ take $H_1=\sigma+3f$ and $D_1=f$.
\item $g\geq2$:  $X_g$ has Picard rank $1$ with ample generator $H_g$ of genus $g$; take $D_g=H_g$.
\end{itemize}

Denote by $M(r,D_g,k)$ the moduli space of $H_g$-stable rank $r$ sheaves $\E$ on $X_g$ with $\chone E=D_g$ and $\chtwo{E}.[X_g]=k$---in the notation of $\S\ref{moduli of sheaves}$, this is $M(v)$ for $v(\E)=(r,D_g,k)$.  Define infinite matrices $\Mat=(M(g)_{ij})_{i,j\geq 0}$ and $\Smat{n}=(\Syst^n(g)_{ij})_{i,j\geq 0}$ of Hodge polynomials by 
\[M(g)_{ij}=\begin{cases}\hodge{M\left(\frac{i-j}{2},D_g,\frac{i+j}{2}\right)},& i -j\equiv 0\mod 2\\0,&i-j\equiv 1\mod 2 \end{cases}\]
\[\Syst^n(g)_{ij}=\begin{cases}\hodge{\Syst^n\left(\frac{i-j}{2},D_g,\frac{i+j}{2}\right)},& i -j\equiv 0\mod 2\\0,&i-j\equiv 1\mod 2 \end{cases}\]
Thus, $M(g)_{ij}$ records the Hodge polynomial of the moduli space of sheaves $\E$ with $i=\chi(\E)$ and $j=\chtwo{\E}$.  In the computations below, it is enough to consider $i,j$ nonnegative.

Recall from \S\ref{stratification} that in this case
$M(r,D,a)_i$ is the stratum of $M(r,D,a)$ of sheaves $\E$ with
$h^0(\E)=i$.  By Remark \ref{generic piece} the highest dimensional stratum is $i=r+a=\chi(\E)$; define a matrix $\Gmat=(M^0(g)_{ij})_{i,j\geq 0}$ of the virtual Hodge polynomials of these generic strata:
\[M^0(g)_{ij}=\begin{cases}\hodge{M\left(\frac{i-j}{2},D_g,\frac{i+j}{2}\right)_{i}},& i -j\equiv 0\mod 2\\0,&i-j\equiv 1\mod 2 \end{cases}\]

\subsection{Encoding the Geometry}\label{encoding}
In the following arguments, we will at any one time be considering $X=X_g$ for a fixed $g$, so we drop the $g$ subscripts from the notation.  

For any locally closed stratification of a scheme $Y$, the virtual Hodge polynomial of $Y$ is the sum of the virtual Hodge polynomials of the strata.  In particular,
\begin{eqn}\label{4}\hodge{M(r,D,a)}=\sum_{i=0}^\infty \hodge{M(r,D,a)_{i}}\end{eqn}Of course the terms are zero for $i<\min(0,r+a)$ and also for $i>>0$.  Similarly
\[\hodge{\Syst^n(r,D,a)}=\sum_{i=0}^\infty \hodge{\Syst^n(r,D,a)_{i}}\]
Recall from \S4.5 that there is a diagram for $r\geq n$, $r+a\geq 2n$,
\[\xymatrix{
&{\Syst^n(r,D,a)_i}\ar[dr]^q\ar[dl]_p&\\
{M(r,D,a)_i}&&{M(r-n,D,a-n)_{i-n}}
}\]
which can be rewritten  as
\[\xymatrix{
&\Syst^n(r+n,D,a+n)_{i+n}\ar[dr]^q\ar[dl]_p&\\
M(r+n,D,a+n)_{i+n}&&M(r,D,a)_{i}\\
}\]
for any $i,r,n\geq 0$ and any $a$.  Recall that the fiber of $p$ above $M(r+n,D,a+n)_{i+n}$ is $\Gr(n,i+n)$ and the fiber of $q$ over $M(r,D,a)_i$ is $\Gr(n,i-r-a)$ (or empty unless $n\leq i-r-a$).  When $M(r,D,a)_i$ is nonepmty, we have $i\geq r+a$ by Remark \ref{generic piece}.  Taking $n=i-r-a$,
\[\xymatrix{
&\Syst^{i-r-a}(i-a,D,i-r)_{2i-r-a}\ar[dr]_\cong^q\ar[dl]_p&\\
M(i-a,D,i-r)_{2i-r-a}&&M(r,D,a)_{i}\\
}\]
where $q$ is an isomorphism and $p$ is an \'{e}tale-locally trivial fibration with fiber $\Gr(i-r-a,2i-r-a)$.  This diagram is valid for any $r\geq 0$, any $a$, and $i\geq a+r$.

For any Zariski-locally trivial fibration $Y\into S$ with fiber $F$---i.e., Zariski-locally trivially on $S$, $Y\into S$ is isomorphic to the projection $F\times S\into S$---the Hodge polynomials simply multiply
\[\hodge{Y}=\hodge{F}\hodge{S}\]
The same is not in general true for \'etale-locally trivial fibrations, but it is in this case:
\begin{lemma}Let $Y,S$ be quasiprojective varieties over $\C$, and $\pi:Y\into S$ a projective \'{e}tale-locally trivial fibration with fiber $\Gr(k,n)$.  Then
\[\hodge{Y}=\hodge{\Gr(k,n)}\hodge{S}\]
\end{lemma}
\begin{proof}  Let $\Omega=\Omega_{Y/S}$ be the relative cotangent bundle, and let $A\subset H_c^*(Y,\Q)$ be the sub-Hodge structure generated by the Chern classes $\mathrm{c}_i(\Omega_{Y/S})$ and their products.  For each fiber $i:\Gr(k,n)\into Y$, $i^*$ clearly restricts to an isomorphism $A\xrightarrow{\cong}H^*(\Gr(k,n),\Q)$ of Hodge structures.  Let $\phi:H^*(\Gr(k,n),\Q)\into A$ be the inverse, and define a morphism of Hodge structures
\[\psi=\phi\smallsmile \pi^*:H^*(\Gr(k,n),\Q)\otimes H_c^*(S,\Q)\into H_c^*(Y,\Q)\]
By the Leray-Hirsch theorem, this is an isomorphism of vector spaces, and therefore of Hodge structures.
\end{proof}
Thus,
\begin{eqnarray*}\hodge{M(r,D,a)_i}&=&\hodge{\Gr(i-r-a,2i-r-a)}\hodge{M(i-a,D,i-r)_{2i-r-a}}\\
&=&\sqnom{2i-r-a}{i-r-a}\hodge{M(i-a,D,i-r)_{2i-r-a}}\\\end{eqnarray*}
After replacing $\ell=r$, $a=k+\ell$, and $i=k+2\ell+s$, this becomes
\begin{equation}\label{gentogen}\hodge{M(\ell,D,k+\ell)_{k+2\ell+s}}=\sqnom{k+2\ell+2s}{s}\hodge{M(\ell+s,D,k+\ell+s)_{k+2\ell+2s}}
\end{equation}
This equation is valid in particular for any $k,\ell,s\geq 0$.  The Hodge polynomial on the right is $M^0(g)_{k+2\ell+2s,k}$.  The strata $M(\ell,D,k+\ell)_{k+2\ell+s}$ are null for $s<0$, so 
\begin{eqnarray*}
M(g)_{k+2\ell,k}&=&\hodge{M(\ell,D,k+\ell)}\\
&=&\sum_{s=0}^\infty \sqnom{k+2\ell+2s}{s}M^0(g)_{k+2\ell+2s,k}\\
&=&\sum_{s=0}^\infty A^0_{k+2\ell,k+2\ell+2s}M^0(g)_{k+2\ell+2s,k}\\
\end{eqnarray*}
where $\Amat{0}=(A^0_{ij})_{i,j\geq0}$ is the matrix from $\S 4.4$.
Thus
\begin{equation}\label{genmatrices}\Mat=\Amat{0}\Gmat\end{equation}
Moreover, since
\[\hodge{\Syst^n(r,D,a)_i}=\hodge{\Gr(n,i)}\hodge{M(r,D,a)_i}\]
We have
\[\hodge{\Syst^n(\ell,D,k+\ell)_{k+2\ell+s}}=\sqnom{k+2\ell+s}{n}\hodge{M(\ell,D,k+\ell)_{k+2\ell+s}}\]
so that
\begin{eqnarray*}
\Syst^n(g)_{k+2\ell,k}&=&\hodge{\Syst^n(\ell,D,k+\ell)}\\
&=&\sum_{s=0}^\infty \sqnom{k+2\ell+s}{n}\sqnom{k+2\ell+2s}{s}M^0(g)_{k+2\ell+2s,k}\\
&=&\sum_{s=0}^\infty A^n_{k+2\ell,k+2\ell+2s}M^0(g)_{k+2\ell+2s,k}\\
\end{eqnarray*}
where $\Amat{n}=(A^n_{ij})_{i,j\geq0}$ is the more general $A$-matrix from $\S 4.4$.  Thus, 
\[\Smat{n}=\Amat{n}\Gmat\]
and setting $\Pmat{n}=\Amat{n}\Amat{0}^{-1}$,
\begin{proposition}\label{product prop}\[\Smat{n}=\Pmat{n}\Mat\]
\end{proposition}The entries of $\Pmat{n}$ are computed in \S\ref{product}.

\subsection{Explicit Computations}\label{explicit}By Theorem \ref{generic}, $M(r,D,a)$ is deformation equivalent to the Hilbert scheme of points $X^{[g-ra]}$, so
\[\hodge{M(r,D,a)}=\hodge{X^{[g-ra]}}\] 
The generating function for the Hodge polynomials of the $X^{[n]}$ is, by G\"{o}ttsche's formula \cite{gottsche},
\begin{align*}
\sum_{n\geq 0}&\hodge{X^{[n]}}q^n=\prod_{n\geq1}\prod_{i,j=0}^2(1-(-1)^{i+j}t^{i-1}\bar{t}^{j-1}(uq)^n)^{-(-1)^{i+j}h^{i,j}(X)}\\
&=\prod_{n\geq 1}\frac{1}{(1-u^{-1}(uq)^n)(1-t\bar{t}^{-1}(uq)^n)(1-(uq)^n)^{20}(1-\bar{t}t^{-1}(uq)^n)(1-u(uq)^n)}\\
\end{align*}
More concisely,
\begin{align}
\sum_{n\geq 0}&\hodge{X^{[n]}}u^{-n}q^{n}=\notag\\
&=\prod_{n\geq 1}\frac{1}{(1-u^{-1}q^n)(1-t^2u^{-1}q^n)(1-q^n)^{20}(1-t^{-2}uq^n)(1-uq^n)}\label{hilbert}
\end{align}

Denote by $c(n)=\hodge{X^{[n]}}$.
We are interested in the generating function (we suppress the $u$-dependence from the notation)
\begin{align}F^r_n(q,y):=&\sum_{g\geq 0}\sum_{k\in\mathbb Z}\hodge{\Syst^n(r,D_g,k+r)}u^{-g}y^kq^{g-1}\notag\\
=&\sum_{g\geq 0}\sum_{k\geq 0}\hodge{\Syst^n(r,D_g,k+r)}u^{-g}y^{k}q^{g-1}\notag\\
&+\sum_{g\geq 0}\sum_{k< 0}\hodge{\Syst^n(r,D_g,k+r)}u^{-g}y^{k}q^{g-1}\label{last line}
\end{align}
The exponent $g-1$ of $q$ (instead of simply $g$) is customary.  For
$r\leq n$, we know by Proposition \ref{dual} that
\[\Syst^n(r,D,r-k)\cong\Syst^n(n-r,D,n-r+k)\]
and therefore we can write \eqref{last line} as
\[F^r_n(q,y)=\sum_{g\geq 0}\sum_{k\geq 0}\Syst^n(g)_{k+2r,k}u^{-g}y^{k}q^{g-1}
+\sum_{g\geq 0}\sum_{k> 0}\Syst^n(g)_{k+2(n-r),k}u^{-g}y^{-k}q^{g-1}\]

We have
\begin{eqnarray*}\Syst^{n}(g)_{k+2r,k}&=&\sum_{\ell\geq r}P^n_{k+2r,k+2\ell}M(g)_{k+2\ell,k}\\
&=&\sum_{\ell\geq r}P^n_{k+2r,k+2\ell}c\left(g-\ell^2-\ell k\right)\\
\Syst^{n}(g)_{k+2(n-r),k}&=&\sum_{\ell\geq r+n}P^n_{k-2r+2n,k+2\ell}M(g)_{k+2\ell,k}\\
&=&\sum_{\ell\geq r+n}P^n_{k-2r+2n,k+2\ell}c\left(g-\ell^2-\ell k\right)
\end{eqnarray*}
Therefore
\begin{align*}
F^r_n(q,y)=&\sum_{g\geq 0}\sum_{k\in\Z}u^{-g}q^{g-1}y^k\Syst^n(g)_{k+2r,k}\\
=&\sum_{g\geq 0}\sum_{k\geq 0}u^{-g}q^{g-1}y^k\sum_{\ell\geq r} P^n_{k+2r,k+2\ell} c(g-\ell^2-\ell k)\\
&+\sum_{g\geq 0}\sum_{k\geq 1}u^{-g}q^{g-1}y^{-k}\sum_{\ell\geq n-r} P^n_{k-2r+2n,k+2\ell} c(g-\ell^2-\ell k)\\
\end{align*}
and thus
\begin{align}
F^r_n(q,y)&=S(q)\sum_{k\geq 0}\sum_{\ell\geq r}y^ku^{-\ell^2-\ell k}q^{\ell k+\ell^2} P^n_{k+2r,k+2\ell}\label{labelly}\\
&+S(q)\sum_{k\geq 1}\sum_{\ell\geq n-r}y^{-k}u^{-\ell^2-\ell k}q^{\ell k+\ell^2} P^n_{k-2r+2n,k+2\ell}\label{labelly2}\end{align}

where 
\begin{eqnarray*} S(q)&=&\sum_{g\geq 0}c(g)u^{-g}q^{g-1}\\
\end{eqnarray*}
is the generating function of the Hodge polynomials of the Hilbert schemes of points on a K3 surface (again with the customary shift in the $q$ power).  We also know by Lemma \ref{useful} that
\begin{eqnarray*}
P^n_{k+2r,k+2\ell}&=&u^{r(n-r)}u^{\ell^2+\ell k-n\ell-kr}\frac{[k+2\ell]}{[n]}\sqnom{n+\ell-r-1}{n-1}\sqnom{k+\ell+r-1}{n-1}\\
P^n_{k-2r+2n,k+2\ell}&=&u^{r(n-r)}u^{\ell^2+\ell k-n\ell-k(n-r)}\frac{[k+2\ell]}{[n]}\sqnom{\ell+r-1}{n-1}\sqnom{k+\ell-r+n-1}{n-1}
\end{eqnarray*}
Note that the sums in \eqref{labelly}, \eqref{labelly2} make sense for all $\ell\geq 0$ since the terms are zero whenever $\ell< r$ in the first and $\ell<n-r$ in the second sum. 

Write $p=\ell+k$ to get 
\begin{align*}
\sum_{k\geq 0}\sum_{\ell\geq 0}&y^ku^{-\ell^2-\ell k}q^{\ell k+\ell^2} P^n_{k+2r,k+2\ell}\notag\\
&=\frac{u^{r(n-r)}}{[n]}\sum_{\ell\geq 0}\sum_{p\geq \ell}u^{-n\ell-(p-\ell)r}[p+\ell]\sqnom{n+\ell-r-1}{n-1}\sqnom{p+r-1}{n-1}y^{p-\ell}q^{p\ell}
\end{align*}
and
\begin{align*}
\sum_{k>0}\sum_{\ell\geq 0}&y^{-k}u^{-\ell^2-\ell k}q^{\ell k+\ell^2} P^n_{k-2r+2n,k+2\ell}\\
&=\frac{u^{r(n-r)}}{[n]}\sum_{p>\ell}\sum_{\ell\geq 0}u^{-np+(p-\ell)r}[p+\ell]\sqnom{\ell+r-1}{n-1}\sqnom{n-r+p-1}{n-1}y^{\ell-p}q^{p\ell}\\
&\stackrel{\ell\leftrightarrow p}{=}\frac{u^{r(n-r)}}{[n]}\sum_{p\geq 0}\sum_{\ell>p}u^{-n\ell-(p-\ell)r}[p+\ell]\sqnom{n+\ell-r-1}{n-1}\sqnom{p+r-1}{n-1}y^{p-\ell}q^{p\ell}
\end{align*}
Where the second line is obtained by setting $k=p-\ell$, and the
third by switching $p$ and $\ell$.  Noting that
$\sqnom{n+\ell-r-1}{n-1}$ and $\sqnom{p+r-1}{n-1}$ vanish for
$\ell<r$ and $p<n-r$, respectively, the result is
\setcounter{theoremletter}{0}
\begin{theorem}\label{thm}For $r\leq n$,
\begin{eqnarray*}
\frac{F^r_n(q,y)}{S(q)}&=&\frac{u^{r(n-r)}}{[n]}\sum_{\substack{p\geq n-r\\\ell\geq r}}u^{-n\ell-(p-\ell)r}[p+\ell]\sqnom{n+\ell-r-1}{n-1}\sqnom{p+r-1}{n-1}y^{p-\ell}q^{p\ell}\\
\end{eqnarray*}
\end{theorem}
\begin{remark}One is able to produce a similar formula for $r>n$ by
  once again using the duality in Prospition \ref{dual}, but it requires defining $M(r,D,a)$ for negative $r$.  Such moduli spaces naturally parametrize objects in the derived category $D^b(X)$ Verdier dual to stable sheaves.
\end{remark}
Note that the only dependence on $t,\bar{t}$ that doesn't factor through $u=t\bar{t}$ is from the term $S(q)$.  In particular for $r=0,n=1$
\begin{align*}
\frac{F^0_1(q,y)}{S(q)}&=\sum_{\ell\geq 0}\sum_{p\geq 1}u^{-\ell}[p+\ell]y^{p-\ell}q^{p\ell}\\
&=\frac{1}{u-1}\sum_{\ell\geq 0}\sum_{p\geq 1}(u^p-u^{-\ell})y^{p-\ell}q^{p\ell}\\
&=\frac{1}{u-1}\Psi(u,y;q)\\
\end{align*}
where $\Psi(u,y;q)$ is the function from $\S 4.3$.  By the computations in $\S \ref{theta}$, we recover 
\begin{corollary}\cite[Theorem 5.158]{KY}.\label{rankone}
\[\frac{F^0_1(q,y)}{S(q)}=\frac{-1}{(1-y)(1-u^{-1}y^{-1})}\prod_{n\geq 1}\frac{(1-q^n)^2(1-uq^n)(1-u^{-1}q^n)}{(1-yq^n)(1-y^{-1}q^n)(1-uyq^n)(1-u^{-1}y^{-1}q^n)}\]
\end{corollary}

For future reference, set
\begin{align}\label{Phi def}\Phi(u,y;q)=\prod_{n\geq 1}\frac{(1-q^n)^2(1-uq^n)(1-u^{-1}q^n)}{(1-yq^n)(1-y^{-1}q^n)(1-uyq^n)(1-u^{-1}y^{-1}q^n)}
\end{align}
Note directly from the formula in Theorem \ref{thm} that the duality in Proposition \ref{dual} manifests itself in the following duality of the generating function $F^r_n(q,y)$:
\begin{corollary}\label{dual2}
\[F^r_n(q,y)=F^{n-r}_n(q,y^{-1})\]
\end{corollary}
\begin{remark}\label{BNcompute}
The same method may be employed to compute the Hodge polynomials of the Brill-Noether strata $M(r,D_g,r+k)_i$ of each moduli space.
\end{remark}

\subsection{Relation to $r=0,n=1$}\label{subsect:KY}
The form of the higher generating functions is strongly determined by the Kawai-Yoshioka ($r=0,n=1$) function.  Define Laurent polynomials $C^r_n(i,j)$ in $u$ for $r\geq 0,n\geq 1$, $1\leq i\leq n$ and $0\leq j\leq n-i$ by $C_n^r(n,0)=1$ and 
\[C^r_{n+1}(i,j)=C_n^r(i-1,j)+C^r_n(i+1,j-1)-u^{r-n}C_n^r(i,j-1)-u^{n-r}C_n^r(i,j)\]

\begin{lemma}The term $\displaystyle u^{-n\ell-(p-\ell)r}[p+\ell]\sqnom{n+\ell-r-1}{n-1}\sqnom{p+r-1}{n-1}$ is equal to
\[\frac{(u-1)^{1-2n}}{[n-1]!^2}\sum_{i=1}^n\sum_{j=0}^{n-i}C^r_n(i,j)(u^{ip}-u^{-i\ell})u^{j(p-\ell)}\]
\end{lemma}
\begin{proof}
Clearly the claim is true for $n=1$.  Note that
\begin{align*}
\frac{u^{-\ell}[n+\ell-r][p+r-n]}{[n]^2}&=\frac{(u-1)^2}{[n]^2}(u^{n-r}-u^{-\ell})(u^{p+r-n-\ell}-u^{\ell})\\
&=\frac{(u-1)^2}{[n]^2}(u^{p}-u^{p+r-n-\ell}-u^{n-r}+u^{-\ell})\\
\end{align*}
Thus by induction
\begin{align}
&u^{-(n+1)\ell-(p-\ell)r}[p+\ell]\sqnom{n+\ell-r-1}{n}\sqnom{p+r-1}{n}\label{the stuff}\\
&=\frac{u^{-\ell}[n+\ell-r][p+r-n]}{[n]^2}\left(u^{-n\ell-(p-\ell)r}[p+\ell]\sqnom{n+\ell-r-1}{n-1}\sqnom{p+r-1}{n-1}\right)\notag\\
&=\frac{(u-1)^2}{[n]^2}(u^{p}-u^{p+r-n-\ell}-u^{n-r}+u^{-\ell})\left(\frac{(u-1)^{2-2n}}{[n-1]!^2}\sum_{i,j}C^r_n(i,j)(u^{ip}-u^{-i\ell})u^{j(p-\ell)}\right)\notag
\end{align}The two fractions match up to give the coefficient we want in front of the sum.
Note that
\begin{align*}(u^p+u^{-\ell})(u^{ip}-u^{-i\ell})u^{j(p-\ell)}&=(u^{(i+1)p}-u^{p-i\ell})u^{j(p-\ell)}+(u^{ip-\ell}-u^{-(i+1)\ell})u^{j(p-\ell)}\\
&= (u^{(i+1)p}-u^{-(i+1)\ell})u^{j(p-\ell)}+(u^{(i-1)p}-u^{-(i-1)\ell})u^{(j+1)(p-\ell)}\\
\end{align*}
and
\[-(u^{p+r-n-\ell}+u^{n-r})(u^{ip}-u^{-i\ell})u^{j(p-\ell)}=-u^{r-n}(u^{ip}-u^{-i\ell})u^{(j+1)(p-\ell)}-u^{n-r}(u^{ip}-u^{-i\ell})u^{j(p-\ell)}\]
So that in \eqref{the stuff} the coefficient of $(u^{ip}-u^{-i\ell})u^{p-\ell}$ is
\[C_n^r(i-1,j)+C^r_n(i+1,j-1)-u^{r-n}C_n^r(i,j-1)-u^{n-r}C_n^r(i,j)\]
which by definition is $C^r_{n+1}(i,j)$.

\end{proof}
By Theorem \ref{thm},
\begin{eqnarray*}
[n]u^{r(r-n)} S(q)^{-1} F^r_n(q,y)&=&\sum_{\ell\geq 0}\sum_{p\geq 0}u^{-n\ell-(p-\ell)r}[p+\ell]\sqnom{n+\ell-r-1}{n-1}\sqnom{p+r-1}{n-1}y^{p-\ell}q^{p\ell}\\
&=&\frac{(u-1)^{1-2n}}{[n-1]!^2}\sum_{i=1}^n\sum_{j=0}^{n-i}C^r_n(i,j)\sum_{p,\ell\geq 0}(u^{ip}-u^{-i\ell})u^{j(p-\ell)}y^{p-\ell}q^{p\ell}\\
&=&\frac{(u-1)^{1-2n}}{[n-1]!^2}\sum_{i=1}^n\sum_{j=0}^{n-i}C^r_n(i,j)\Psi(u^i,u^jy;q)\\
\end{eqnarray*}
So finally
\begin{theorem}\label{modus}
\[\frac{F^r_n(q,y)}{S(q)}=\frac{u^{r(n-r)}(u-1)^{1-2n}}{[n][n-1]!^2}\sum_{i=1}^n\sum_{j=0}^{n-i}C^r_n(i,j)\Psi(u^i,u^jy;q)\]
\end{theorem}
For example, for $n=2$ the only nonzero $C^r_2(i,j)$ are 
\begin{eqnarray*}
C^r_2(2,0)=1&C^r_2(1,0)=-u^{1-r}&C^r_2(1,1)=-u^{r-1}\\
\end{eqnarray*}
and therefore
\begin{equation}\label{special}u^{r(r-2)}(u-1)^{3}[2]S(q)^{-1}F^r_2(q,y)=\Psi(u^2,y)-u^{1-r}\Psi(u,y)-u^{r-1}\Psi(u,uy)\end{equation}
\subsection{Euler Characteristics and Modularity}\label{euler}
Of particular interest is the generating function $f^r_n(q,y):=F^r_n(q,y)|_{t=\bar t=1}$ of the Euler characteristics $\chi\left(\Syst^n(r,D,a)\right)$ of the stable pair moduli spaces.  By definition,
\[f^r_n(q,y)=\sum_{g\geq 0}\sum_{k\in\mathbb Z}\chi\left(\Syst^n(r,D_g,k+r)\right)y^kq^{g-1}\]
The generating function $s(q)$ of the Euler characteristics of the Hilbert scheme of points is well known.  From \eqref{hilbert},
\[s(q)=S(q)|_{t=\bar t=1}=\sum_{g\geq 0}\chi(X^{[g]})q^{g-1}=q^{-1}\prod_{g\geq 1}\frac{1}{(1-q^g)^{24}}=\frac{1}{\eta(q)^{24}}\]
where $\eta(q)$ is the $q$-expansion of the Dedekind $\eta$ function.
Define
\[G^r_n(q,y)=\frac{F^r_n(q,y)}{S(q)}\]
and
\[g^r_n(q,y)=\frac{f^r_n(q,y)}{s(q)}\]
From Theorem \ref{thm},
\begin{theorem}
\[g^r_n(q,y)=\frac{1}{n}\sum_{\substack{p\geq n-r\\\ell\geq r}}(p+\ell)\binom{n+\ell-r-1}{n-1}\binom{p+r-1}{n-1}y^{p-\ell}q^{p\ell}\]
\end{theorem}Note that the coefficient in \eqref{rankone} can be rewritten at $u=1$ as 
\[\frac{-1}{(1-y)(1-y^{-1})}=\left(\sqrt{y}-\frac{1}{\sqrt{y}}\right)^{-2}\]
Thus, for $r=0$, $n=1$ we recover the Kawai-Yoshioka formula \cite{KY}
\begin{corollary}
\[g_1^0(q,y)=\left(\sqrt{y}-\frac{1}{\sqrt{y}}\right)^{-2}\prod_{n\geq 1}\frac{(1-q^n)^4}{(1-yq^n)^2(1-y^{-1}q^n)^2}\]
\end{corollary}
From \cite{MPT} we know the $v$ coefficients of $g_1^0(q,y)$ after the change of variable $y=-e^{iv}$ are (the $q$-expansions of) classical modular forms,
\[-g^0_1(q,-y)\stackrel{y=-e^{iv}}{=}\frac{1}{v^2}\cdot\exp\left(\sum_{g\geq 1}v^{2g}\frac{|B_{2g}|}{g\cdot (2g)!}E_{2g}(q)\right)\]
where $E_{2g}(q)$ is the $q$-expansion of the $2g$th Eisenstein series and $B_{2g}$ is the $2g$th Bernoulli number, defined by
$\frac{t}{e^t-1}=\sum_{n=0}^\infty B_n\frac{t^n}{n!}$.  See \cite{forms}, for example, for an elementary treatment of modular forms.  Note that

\begin{align*}
\frac{iv}{e^{iv}-1} &= \sum_{m\geq 0} \frac{B_m(iv)^m}{m!}\\
\frac{-iv}{e^{-iv}-1} &= \sum_{m\geq 0} \frac{B_m(-iv)^m}{m!}
\end{align*}

thus \begin{align*}
\lim_{u\to 1}\frac{v^2}{(1-u^{k+l}e^{iv})(1-u^{-l}e^{-iv})}&=\sum_{m,n\geq 0}B_m\frac{(iv)^m}{m!}B_n\frac{(-iv)^n}{n!}\\
&=\sum_{n\geq 0}\frac{i^nv^n}{n!}\sum_{k=0}^n(-1)^kB_kB_{n-k}\binom{n}{k}
\end{align*}
is a power series in $\mathbb{Q}[\![v]\!]$, which we denote by $\mathcal{B}$.

The divisor functions
\[\sigma_g(n)=\sum_{d|n}d^g\]
are related to the Eisenstein series by
\[E_{2g}(q)=1-\frac{4g}{B_{2g}}\sum_{n\geq 1}\sigma_{2g-1}q^n\]
$E_{2g}(q)$ is a modular form of weight $2g$ and level $\Gamma(1)$.  The Eisenstein series $E_{2g+1}(q)$ of odd weight $2g+1$ and level $\Gamma(2)$ are defined by
\[E_{2g+1}(q)=1+\frac{4(-1)^g}{e_{2g}}\sum_{n\geq 1}\sigma_{2g-1}q^{n/2}\]
where the numbers $e_n$ are defined by $\frac{1}{\cos t}=\sum_{n\geq 0}e_n\frac{t^n}{n!}$  Let
\[R=\Q(i)[E_{2g}(q),E_{2g+1}(q^2)|g\geq 1]\]
be an algebra generated by modular forms on $\Gamma(4)$.  Clearly the generating functions 
$\displaystyle \Sigma_g=\sum_{n\geq 1}\sigma_g(n)q^n\in R$
for $g\geq 1$.  The modularity result for $g^r_n(q,y)$ is:
\begin{theorem}\label{thm:modularity}
The coefficient of $v^s$ in the power series expansion of $v^2g_n^r(q,e^{iv})$ is itself a power series in 
$q$, and this coefficient is in fact in the algebra $R$.
\end{theorem}
First we have
\begin{lemma}Let $\displaystyle\log\Phi(u^k,u^\ell e^{iv};q)=\sum_{s\geq 0}\psi_{k,\ell,s}v^s$ where $\psi_{k,\ell,s}$ is a function of $u$ and $q$ (recall $\Phi$ was defined in \eqref{Phi def}). Then for all $t\geq 0$, the $t$-th derivatives $\frac{d^t}{du^t}\psi_{k,\ell,s}|_{u=1}\in R$.
\end{lemma}
\begin{proof}
By definition 
\begin{align*}
\Phi(u^k,u^\ell e^{iv},q)&=\prod_{n\geq 1}\frac{(1-q^n)^2(1-u^kq^n)(1-u^{-k}q^n)}{(1-u^{k+\ell}e^{iv}q^n)(1-u^{-k-\ell}e^{-iv}q^n)(1-u^{\ell}e^{iv}q^n)(1-
u^{-\ell}e^{-iv}q^n)}
\end{align*}
and so
\begin{align*}
\log\Phi(u^k,u^\ell e^{iv};q)=&\sum_{n\geq 1}\left(2\log(1-q^n)+\log(1-u^kq^n)+\log(1-u^{-k}q^n)\right.\\
&-\log(1-u^{k+\ell}e^{iv}q^n)+
\log(1-u^{-k-\ell}e^{-iv}q^n)\\
&+\left.\log(1-u^{\ell}e^{iv}q^n)+\log(1-u^{-\ell}e^{-iv}q^n)\right)\\
=&\sum_{n\geq 1}\sum_{r\geq 1}\frac{q^{nr}}{r}\left(2+u^{kr}+u^{-kr}\right)-\\
&\sum_{n\geq 1}\sum_{r\geq 1}\sum_{s\geq 0}\frac{q^{nr}(ivr)^s}{rs!}\left(u^{(k+\ell)r}+(-1)^su^{-(k+\ell)r}+u^{\ell r}+(-1)^su^{-\ell r}\right)\\
=&\sum_{n\geq 1}q^n\sum_{r\mid n}\frac{\left(2+u^{kr}+u^{-kr}\right)}{r}-\\
&\sum_{s\geq 0}\frac{i^sv^s}{s!}\sum_{n\geq 1}q^n\sum_{r\mid n}r^{s-1}\left(u^{(k+\ell)r}+(-1)^su^{-(k+\ell)r}+u^{\ell r}+(-1)^su^{-\ell r}\right)
\end{align*}
This implies that
\[\psi_{k,\ell,0}=\sum_{n\geq 1}q^n\sum_{r\mid n}\frac{\left(2+u^{kr}+u^{-kr}-u^{(k+\ell)r}-u^{-(k+\ell)r}-u^{\ell r}-u^{-\ell  r}\right)}{r}\]
and for $s\geq 1$
\[\psi_{k,\ell,s}=-\frac{i^s}{s!}\sum_{n\geq 1}q^n\sum_{r\mid n}r^{s-1}\left(u^{(k+\ell)r}+(-1)^su^{-(k+\ell)r}+u^{\ell r}+(-1)^su^{-\ell r}\right)\]
Evaluating at $u=1$ we get $\psi_{k,\ell,0}|_{u=1}=0$ and $\displaystyle \psi_{k,\ell,s}|_{u=1}=-\frac{2(1+(-1)^2)i^s}{s!}\Sigma_{s-1}$. Differentiating, we get that for $t\geq 1$ and $s\geq 1$ we have
\begin{align*}
\left.\left(\frac{d^t}{du^t}\psi_{k,\ell,0}\right)\right|_{u=1}&=t!\sum_{n\geq 1}q^n\sum_{r\mid n}r^{s-1}\left(\binom{kr}{t}+\binom{-kr}{t}-\binom{(k+\ell)r}{t}\right.\\
&-\left.\binom{-(k+\ell)r}{t}-\binom{\ell r}{t}-\binom{-\ell r}{t}\right)
\end{align*}
and
\begin{align*}
\left.\left(\frac{d^t}{du^t}\psi_{k,\ell,s}\right)\right|_{u=1}&=-\frac{i^st!}{s!}\sum_{n\geq 1}q^n\sum_{r\mid n}r^{s-1}\left(\binom{(k+\ell)r}{t}+(-1)^s\binom{-(k+\ell)r}{t}\right.\\
&=+\left.\binom{\ell r}{t}+(-1)^s\binom{-\ell r}{t}\right)
\end{align*}
The conclusion then follows as each coefficient of $q^n$ in the above expansions is either 0 or a $\mathbb{Q}$-linear combination of 
powers of $r$ which implies that the derivative evaluated at $u=1$ is a linear combination of terms of the form $\Sigma_w$ for $w\geq 1$.
\end{proof}

\begin{proof}[Proof of Theorem \ref{thm:modularity}]
Note that
\begin{align*}
v^2g_n^r(q,e^{iv})&=\lim_{u\to 1}v^2G_n^r(q,e^{iv})\\
&=\lim_{u\to 1}\frac{u^{r(n-r)}(u-1)^{1-2n}}{[n][n-1]!^2}\sum_{k=1}^n\sum_{\ell=0}^{n-k}C^r_n(i,j)\Psi(u^k,u^\ell y;q)\\
&=\mathcal{B}\lim_{u\to 1}\frac{u^{r(n-r)}(u-1)^{1-2n}}{[n][n-1]!^2}\sum_{k=1}^n\sum_{\ell=0}^{n-k}C^r_n(i,j)\Phi(u^k,u^\ell y;q)
\end{align*}
To compute the limit we apply L'H\^opital observing that 
\[\left.\frac{d^{2n-1}}{du^{2n-1}}\frac{[n][n-1]!^2}{(u-1)^{2n-1}}\right|_{u=1}=n^2\] We get
\begin{align*}
v^2g_n^r(q,e^{iv})&=\left.\frac{\mathcal{B}}{n^2}\frac{d^{n^2}}{du^{n^2}}
\left(\sum_{k=1}^n\sum_{\ell=0}^{n-k}C^r_n(k,\ell)\Phi(u^k,u^\ell e^{iv};q)\right)\right|_{u=1}
\end{align*}
so it is enough to check that for all $t\geq 0$, $\frac{d^t}{du^t}\Phi(u^k,u^\ell e^{iv};q)|_{u=1}\in R[\![v]\!]$.

But $\displaystyle \frac{d^t}{du^t}\Phi(u^k,u^\ell e^{iv};q)=\frac{d^t}{du^t}\exp\left(\sum_{s\geq 0}\psi_{k,\ell,s}v^s\right)$ is of the form
\[\exp\left(\sum_{s\geq 0}\psi_{k,\ell,s}v^s\right)\mathcal{F}_{k,\ell,t}=\Phi(u^k,u^\ell e^{iv};q)\mathcal{F}_{k,\ell,t}\]
where $\mathcal{F}_{k,\ell,t}$ is an expression involving only the $\psi_{k,\ell,s}$ and their derivatives. Evaluating at $u=1$, the previous lemma shows that all coefficients of powers of $v$ in $\mathcal{F}_{k,\ell,t}$ are in $R$. Finally, note that 
\[\Phi(1, e^{iv}; q)=\prod_{n\geq 1}\frac{(1-q^n)^4}{(1-e^{iv}q^n)^2(1-e^{-iv}q^n)^2}\]
and this was computed in [MPT, p. 53] to be $\displaystyle 4\sum_{k\geq 1}\frac{(-1)^kv^{2k}}{(2k)!}\Sigma_{2k-1}$. Multiplying everything together we get the required conclusion.
\end{proof}

\section{Computations}\label{sect:4}
\newcommand{\ps}[1]{[\![#1]\!]}
\subsection{$u$-Binomial Coefficients}\label{uint}
The $u$-integer $[n]$ is the polynomial in $u$ given by 
\[[n]=\frac{u^{n}-1}{u-1}\]
The $u$-factorial and $u$-binomial coefficients are defined similarly:
\begin{align*}
[n]!=\prod_{s=1}^{n}[s]&&\sqnom{n}{k}=\begin{cases}\frac{[n]!}{[k]![n-k]!}&k\leq n\\
0&k>n\\ 
\end{cases}\\
\end{align*}
By fiat $[0]!=1$.  

\subsection{Properties of $u$-Binomial Coefficients}\label{ubi}
Most binomial identities have $u$-analogs, many of which recover the classical identities in the $u\into 1$ limit.  We collect here the properties we will need with proofs.
\begin{lemma}\label{refer}For any $k\leq n$\begin{enumerate}
\item 
\[[n]=[n-k]+u^{n-k}[k]\]
\item \begin{eqn}\label{hello}\sqnom{n+1}{k}=\sqnom{n}{k}+u^{n+1-k}\sqnom{n}{k-1}\end{eqn}

\end{enumerate}
\end{lemma}
\begin{proof}
\begin{enumerate}
\item Follows immediately from $[n+1]=\sum_{s=0}^{n}u^s$.
\item\begin{eqnarray*}\sqnom{n+1}{k}&=&\frac{[n+1]!}{[k]![n+1-k]!}\\
&=&\frac{[n]!}{[k]![n-k]!}\left(\frac{[n+1]}{[n+1-k]}\right)\\
&=&\frac{[n]!}{[k]![n-k]!}\left(1+u^{n+1-k}\frac{[k]}{[n+1-k]}\right)\\
&=&\sqnom{n}{k}+u^{n+1-k}\sqnom{n}{k-1}\\
\end{eqnarray*}
\end{enumerate}
\end{proof}Note that $\sqnom{n}{k}$ has degree $k(n-k)$.  The symmetric $u$-binomial coefficient is defined for $0\leq k\leq n$ by
\[\curbnom{n}{k}=u^{-\frac{k(n-k)}{2}}\sqnom{n}{k}\]
Also, under the same conditions let
\[\curbnom{-n}{k}=(-1)^k\curbnom{n+k-1}{k}\]
Let 
\[K_n(t,u)=\prod_{s=0}^{n-1}(1+tu^{s-\frac{n-1}{2}})\]
for $n\geq 0 $.  
\begin{lemma}
\[K_n(t^{-1},u)=t^{-n}K_n(t,u)\]
\end{lemma}
\begin{proof}
\begin{align*}
K_n(t^{-1},u)&=t^{-n}\prod_{s=0}^{n-1}(t+u^{s-\frac{n-1}{2}})
\end{align*}
but terms in the product come in pairs $(t+u^{s})(t+u^{-s})=(1+tu^s)(1+tu^{-s})$.
\end{proof}$K_n$ is invertible as a Laurent series in $t,u^{\frac{1}{2}}$; let 
\[K_{-n}(t,u)=K_{n}(t,u)^{-1}\]
There is an analog of Lemma \ref{refer} for symmetric $u$-binomial coefficients:

\begin{lemma}\label{candostuff}For any $0\leq k\leq n$\begin{enumerate}
\item 
\begin{eqn}\label{1}\curbnom{n+1}{k}=u^{-\frac{k}{2}}\curbnom{n}{k}+u^{\frac{n+1-k}{2}}\curbnom{n}{k-1}\end{eqn}

\item
$K_{n}(t,u)$ is the generating function for the $\curbnom{n}{k}$, that is
\[K_{n}(t,u)=\sum_{k=0}^{\infty}t^k\curbnom{n}{k}\]
\item
\[\curbnom{n+k}{k}=\sum_{s=0}^{k}u^{\frac{sn+s-k}{2}}\curbnom{n+k-s-1}{k-s}\]
\item 
$K_{-n}(t,u)$ is the generating function for the $\curbnom{-n}{k}$, that is
\[K_{-n}(t,u)=\sum_{k=0}^{\infty}t^k\curbnom{-n}{k}\]
\end{enumerate}
\end{lemma}
\begin{proof}
\begin{enumerate}
\item Multiplying \eqref{hello} by $u^{\frac{k(n+1-k)}{2}}$ gives \eqref{1}.
\item

Note that 
\begin{eqn}\label{2}K_{n+1}(t,u)=\left(1+tu^{\frac{n}{2}}\right)K_{n}(tu^{-\frac{1}{2}},u)
\end{eqn}
Assuming by induction that the coefficient of $t^s$ in $K_{n}(tu^{-\frac{1}{2}},u)$ is $u^{-\frac{s}{2}}\curbnom{n}{s}$, the coefficient of $t^k$ in $K_{n+1}(t,u)$ is 
\[u^{-\frac{k}{2}}\curbnom{n}{k}+u^{\frac{n-k+1}{2}}\curbnom{n}{k-1}\]
which yields the result given part (1).
\item 
Replacing $n$ in \eqref{1} with $n+k-1$ we have
\begin{eqn}\label{3}\curbnom{n+k}{k}=u^{-\frac{k}{2}}\curbnom{n+k-1}{k}+u^{\frac{n}{2}}\curbnom{n+k-1}{k-1}\end{eqn}
Note that
\begin{align*}
\sum_{s=0}^{k}u^{\frac{sn+s-k}{2}}&\curbnom{n+k-s-1}{k-s}=u^{-\frac{k}{2}}\curbnom{n+k-1}{k}+\sum_{s=1}^{k}u^{\frac{sn+s-k}{2}}\curbnom{n+k-s-1}{k-s}\\
&=u^{-\frac{k}{2}}\curbnom{n+k-1}{k}+u^{\frac{n}{2}}\left(\sum_{s=0}^{k-1}u^{\frac{sn+s-k+1}{2}}\curbnom{n+k-s-2}{k-s-1}\right)\\
\end{align*}
By induction the term in parentheses is $\curbnom{n+k-1}{k-1}$, and by \eqref{3} the result follows.
\item
Inverting \eqref{2}, we have
\[K_{-n-1}(t,u)=\frac{1}{1+tu^{\frac{n}{2}}}K_{-n}(tu^{-\frac{1}{2}},u)=K_{-n}(tu^{-\frac{1}{2}},u)\sum_{s=0}^\infty (-1)^st^su^{\frac{ns}{2}}\]
Inductively assuming the coefficient of $t^{k-s}$ in $K_{-n}(tu^{-\frac{1}{2}},u)$ is 

\[u^{-\frac{k-s}{2}}\curbnom{-n}{s}=(-1)^{k-s}u^{-\frac{k-s}{2}}\curbnom{n+k-s-1}{k-s}\] the coefficient of $t^k$ in $K_{-n-1}(t,u)$ is
\[(-1)^k\sum_{s=0}^ku^{\frac{ns+s-k}{2}}\curbnom{n+k-s-1}{k-s}=(-1)^k\curbnom{n+k}{k}=\curbnom{-n-1}{k}\]
by part (3).

\end{enumerate}
\end{proof}

\subsection{$q$-Theta Functions}\label{theta}
Given expressions $a,b$ polynomial in $q$ (we will be more precise below), the Pochhammer symbol $(a,b)_{\infty}$ is a formal power series in $q$ defined by
\[(a,b)_{\infty}=\prod_{n=0}^\infty(1-ab^n)\] 
For example, $(q,q)_{\infty}=\prod_{n\geq 1}(1-q^n)$.
The $q$-theta function $\Theta(x;q)\in \Q[x,x^{-1}][[q]]$ is a formal power series in $q$ whose coefficients are Laurent polynomials in $x$.  It is defined by
\[\Theta(x;q)=(q,q)_{\infty}(x,q)_{\infty}(x^{-1}q,q)_{\infty}=(1-x)\prod_{n=1}^{\infty}(1-q^n)(1-xq^n)(1-x^{-1}q^n)\]
In particular $\Theta(x;q)$ has a simple root at $x=1$.  Our main use for $\Theta(x;q)$ is derived from an identity involving
\[\Phi(a,b;q):=\frac{(q,q)^3_{\infty}\Theta(ab;q)}{\Theta(a;q)\Theta(b;q)}\]
Note $\Phi(a,b;q)$ is not an element of $\Q[a,b]\ps{q}$, but it converges for $|q|<|a|,|b|<1$.  We have
\begin{lemma}For $n\in\Z$, define
\[\sign(n)=\begin{cases}
+1&n\geq 0\\
-1& n<0\\
\end{cases}\]
Then 
\[\Phi(a,b;q)=\sum_{\sign(i)=\sign(j)}\sign(i)a^ib^jq^{ij}\]
for $|q|<|a|,|b|<1$.
\end{lemma}
\begin{proof}
See \cite[Theorem 1.5]{mock}.
\end{proof}Define
\[\Psi(x,y;q)=\sum_{\ell\geq 0}\sum_{p\geq 1}(x^p-x^{-\ell})y^{p-\ell}q^{p\ell}\]
The actual statement we needed in \S\ref{explicit} is
\begin{lemma}As formal power series
\[\Psi(x,y;q)=\Phi(xy,y^{-1};q)\]
\end{lemma}
\begin{proof}By \cite[Theorem 1.4]{mock},
\[\sum_{p\in\mathbb{Z}}\frac{a^p}{1-q^pb}=\Phi(a,b;q)\] for $0<|q|<|a|<1$ and $b\neq q^{p}$ for any $p\in\Z$.  On the region 
\[R=\{(q,x,y)\in\C^3|0<|q|<|x|<|y^{-1}|<1\}\] 
we have, for $p>0$,  $|q^py|<1$, and for $p\geq 0$, $|q^py^{-1}|<1$.  Thus, each line in the following converges in $R$:
\begin{align*}
\Phi(xy,y^{-1};q)&=\sum_{p> 0}\frac{(xy)^p}{1-q^py^{-1}}+\frac{1}{1-y^{-1}}+\sum_{p<0}\frac{(xy)^p}{1-q^py^{-1}}\\
&=\sum_{p> 0}\frac{(xy)^p}{1-q^py^{-1}}+\frac{1}{1-y^{-1}}+\sum_{p>0}\frac{(xy)^{-p}}{1-q^{-p}y^{-1}}\\
&=\sum_{p> 0}\frac{(xy)^p}{1-q^py^{-1}}-\frac{y}{1-y}-\sum_{p>0}\frac{(q^py)(xy)^{-p}}{1-q^{p}y}\\
&=\sum_{p> 0}\sum_{\ell\geq 0}(xy)^pq^{p\ell}y^{-\ell}-\sum_{p>0}\sum_{\ell\geq 0}(q^py)(xy)^{-p}q^{p\ell}y^\ell-\frac{y}{1-y}\\
&\stackrel{(*)}{=}\sum_{p>0}\sum_{\ell\geq 0}(xy)^pq^{p\ell}y^{-\ell}-\sum_{\ell>0}\sum_{p>0}(xy)^{-\ell}q^{p\ell}y^{p}-\frac{y}{1-y}\\
&=\sum_{p>0}\sum_{\ell>0}(xy)^pq^{p\ell}y^{-\ell}-\sum_{\ell>0}\sum_{p>0}(xy)^{-\ell}q^{p\ell}y^{p}+\frac{xy}{1-xy}-\frac{y}{1-y}\\
&=\sum_{p,\ell>0}(x^p-x^{-\ell})y^{p-\ell}q^{p\ell}+\frac{xy}{1-xy}-\frac{y}{1-y}\\
\end{align*}
In the equality labeled (*) we replaced $\ell+1\mapsto p$ and $p\mapsto \ell$.  Thus, on $R$ we have
\begin{align*}
\sum_{p,\ell>0}&(x^p-x^{-\ell})y^{p-\ell}q^{p\ell}+\frac{xy}{1-xy}-\frac{y}{1-y}
=\frac{(q,q)_\infty\Theta(x;q)}{\Theta(xy;q)\Theta(y^{-1};q)}\\
&=\frac{(1-x)}{(1-xy)(1-y^{-1})}\prod_{n\geq 1}\frac{(1-q^n)^2(1-xq^n)(1-x^{-1}q^n)}{(1-xyq^n)(1-x^{-1}y^{-1}q^n)(1-yq^n)(1-y^{-1}q^n)}
\end{align*}
which can be rewritten as 
\begin{align}(1-xy)&(1-y^{-1})\left(\sum_{p,\ell>0}(x^p-x^{-\ell})y^{p-\ell}q^{p\ell}+\frac{xy}{1-xy}-\frac{y}{1-y}\right)\notag\\
&=(1-x)\prod_{n\geq 1}\frac{(1-q^n)^2(1-xq^n)(1-x^{-1}q^n)}{(1-xyq^n)(1-x^{-1}y^{-1}q^n)(1-yq^n)(1-y^{-1}q^n)}\label{nasty}\end{align}
For any $x,y$ with $|x|<|y^{-1}|$, \eqref{nasty} is an equality of series in $\mathbb{C}\ps{q}$ converging for $|q|<|x|$.  Therefore it must be an equality of formal power series in $\mathbb{C}[x,y,x^{-1},y^{-1}]\ps{q}$. Since both sides converge for $|q|,|xy|,|y|<1$ it follows it must be an equality of series in $\mathbb{C}\ps{q}$ for any such $x,y$; therefore, in that case, it must be that $\displaystyle (1-xy)(1-y^{-1})\sum_{p>0,\ell\geq 0}(x^p-x^{-\ell})y^{p-\ell}q^{p\ell}$ is equal to
\[(1-x)\prod_{n\geq 1}\frac{(1-q^n)^2(1-xq^n)(1-x^{-1}q^n)}{(1-xyq^n)(1-x^{-1}y^{-1}q^n)(1-yq^n)(1-y^{-1}q^n)}\]
and the conclusion follows.
\end{proof}

\subsection{A Useful Matrix}\label{matrix}
In $\S \ref{encoding}$ we used the matrix $\Amat{n}=(A^n_{ij})_{i,j\geq0}$ defined by
\[A^n_{ij}=\begin{cases}\sqnom{\frac{i+j}{2}}{n}\sqnom{j}{\frac{j-i}{2}}& i -j\equiv 0\mod 2\\0&i-j\equiv 1\mod 2 \end{cases}\]
i.e., the only nonzero entries are $A^n_{k,k+2\ell}=\sqnom{k+\ell}{n}\sqnom{k+2\ell}{\ell}$, $k,\ell\geq 0$.  In particular, $A^0_{k,k+2\ell}=\sqnom{k+2\ell}{\ell}$.  $\Amat{0}$ is upper triangular with ones along the diagonal, and is therefore invertible:
\begin{proposition}The inverse of $\Amat{0}$ is the matrix $\mathbf{B}=(B_{ij})_{i,j\geq 0}$ given by
\[B_{k,k+2\ell}=(-1)^\ell u^{\binom{\ell}{2}}\frac{[k+2\ell]}{[k+\ell]}\sqnom{k+\ell}{\ell}\]
and $B_{k,k+2\ell+1}=0$, for $k,\ell\geq 0$
\end{proposition}
\begin{proof}We need only check that the $(k,k+2\ell)$ entry of $\Amat{0}\mathbf{B}$ for $\ell>0$ is $0$, since the diagonal terms are clearly $1$ and both matrices are upper triangular.  The relevant entries of $\mathbf{B}$ are
\[B_{k+2s,k+2\ell}=(-1)^{\ell-s}u^{\binom{\ell-s}{2}}\sqnom{k+\ell+s}{\ell-s}\frac{[k+2\ell]}{[k+\ell+s]}\]
Also note that
\begin{align*}\sqnom{k+2s}{s}\sqnom{k+\ell+s}{\ell-s}\frac{[k+2\ell]}{[k+\ell+s]}&=
\left(\frac{[k+2s]!}{[s]![k+s]!}\right)\left(\frac{[k+s+\ell]!}{[\ell-s]![k+2s]!}\right)\frac{[k+2\ell]}{[k+\ell+s]}\\
&=\left(\frac{[\ell]!}{[s]![\ell-s]!}\right)\left(\frac{[k+s+\ell-1]!}{[k+s]![\ell-1]!}\right)\frac{[k+2\ell]}{[\ell]}\\
&=\sqnom{\ell}{s}\sqnom{k+\ell-1}{\ell-1}\frac{[k+2\ell]}{[\ell]}\\
\end{align*}
Thus
\begin{align*}\sum_{s=0}^\infty A^0_{k,k+2s}&B_{k+2s,k+2\ell}=\sum_{s=0}^\ell(-1)^{\ell-s}u^{\binom{\ell-s}{2}}\sqnom{k+2s}{s}\sqnom{k+\ell+s}{\ell-s}\frac{[k+2\ell]}{[k+\ell+s]}\\
&=\left(\frac{[k+2\ell]}{[\ell]}\right)\sum_{s=0}^\ell(-1)^{\ell-s}u^{\binom{\ell-s}{2}}\sqnom{k+s+\ell-1}{\ell-1}\sqnom{\ell}{s}\\
&=\left(\frac{[k+2\ell]}{[\ell]}\right)\sum_{s=0}^\ell(-1)^{\ell-s}u^{\binom{\ell-s}{2}+\frac{(\ell-1)(k+s)}{2}+\frac{s(\ell-s)}{2}}\curbnom{k+s+\ell-1}{\ell-1}\curbnom{\ell}{s}\\
&=u^{\frac{\ell^2-\ell+(\ell-1)k}{2}}\left(\frac{[k+2\ell]}{[\ell]}\right)\sum_{s=0}^\ell(-1)^{\ell-s}\curbnom{k+s+\ell-1}{\ell-1}\curbnom{\ell}{s}\\
&=(-1)^{k+\ell}u^{\frac{\ell^2-\ell+(\ell-1)k}{2}}\left(\frac{[k+2\ell]}{[\ell]}\right)\sum_{s=0}^\ell\curbnom{-\ell}{k+s}\curbnom{\ell}{s}\\
\end{align*}
By (4) of $\S\ref{candostuff}$, $\curbnom{-\ell}{k+s}$ is the coefficient of $t^{k+s}$ in $K_{-\ell}(t,q)$ and $\curbnom{\ell}{s}$ is the coefficient of $t^{-s}$ in $K_{\ell}(t^{-1},q)$. Therefore, the sum is the coefficient of $t^k$ in $K_{-\ell}(t,q)K_{\ell}(t^{-1},q)=t^{-\ell}K_{-\ell}(t,q)K_{\ell}(t,q)=t^{-\ell}$ so it must be $0$, unless $\ell=k=0$, but we assumed $\ell>0$.
\end{proof}
\subsection{A Useful Product}\label{product}
In $ \S\ref{explicit}$, an explicit computation of the product $\Pmat{n}:=\Amat{n}\Amat{0}^{-1}$ enabled us to perform the calculation.  The product matrix $\Pmat{n}=(P^n_{ij})_{i,j\geq 0}$ is given by
\begin{lemma}\label{useful}For $k,\ell\geq 0$, $n>0$,
\[P^n_{k,k+2\ell}=u^{\ell^2+\ell (k-n)}\frac{[k+2\ell]}{[n+\ell]}\sqnom{n+\ell}{n}\sqnom{k+\ell-1}{n-1}\]
and $P^n_{k,k+2\ell+1}=0$.
\end{lemma}
\begin{proof}The proof is a calculation very similar to the proof of Lemma \ref{candostuff}.  Note that for $\ell\geq s$
\begin{align*}
\sqnom{k+s}{n}&\sqnom{k+2s}{s}\sqnom{k+s+\ell}{\ell-s}=\\
&=\frac{[k+s]\cdots[k+s-n+1]}{[n]!}\frac{[k+2s]\cdots[k+s+1]}{[s]!}\frac{[k+s+\ell]\cdots[k+2s+1]}{[\ell-s]!}\\
&=\frac{[k+s+\ell]!}{[n]![s]![\ell-s]![k+s-n]!}\\
&=\left(\frac{[n+\ell]!}{[n]![\ell]!}\right)\left(\frac{[\ell]!}{[s]![\ell-s]!}\right)\left(\frac{[k+s+\ell-1]!}{[k+s-n]![n+\ell-1]!}\right)\frac{[k+s+\ell]}{[n+\ell]}\\
\end{align*}so
\begin{align}P^n_{k,k+2\ell}&=\sum_{s=0}^\ell A^n_{k,k+2s}B_{k+2s,k+2\ell}\notag\\
&=\sum_{s=0}^\ell(-1)^{\ell-s}u^{\binom{\ell-s}{2}}\sqnom{k+s}{n}\sqnom{k+2s}{s}\sqnom{k+s+\ell}{\ell-s}\frac{[k+2\ell]}{[k+s+\ell]}\notag\\
&=\frac{[k+2\ell]}{[n+\ell]}\sqnom{n+\ell}{n}\sum_{s=0}^\ell(-1)^{\ell-s}u^{\binom{\ell-s}{2}}\sqnom{k+s+\ell-1}{n+\ell-1}\sqnom{\ell}{s}\notag\\
&=\frac{[k+2\ell]}{[n+\ell]}\sqnom{n+\ell}{n}\sum_{s=0}^\ell(-1)^{\ell-s}u^{\binom{\ell-s}{2}+\frac{(n+\ell-1)(k-n+s)}{2}+\frac{s(\ell-s)}{2}}\curbnom{k+s+\ell-1}{n+\ell-1}\curbnom{\ell}{s}\notag\\
&=\frac{[k+2\ell]}{[n+\ell]}\sqnom{n+\ell}{n}u^{\frac{\ell^2-\ell+(n+\ell-1)(k-n)}{2}}\sum_{s=0}^\ell(-1)^{\ell-s}u^{sn/2}\curbnom{k+s+\ell-1}{n+\ell-1}\curbnom{\ell}{s}\notag\\
&=(-1)^{k-n+\ell}\frac{[k+2\ell]}{[n+\ell]}\sqnom{n+\ell}{n}u^{\frac{\ell^2-\ell+(n+\ell-1)(k-n)}{2}}\sum_{s=0}^\ell u^{sn/2}\curbnom{-(n+\ell)}{k-n+s}\curbnom{\ell}{s}\label{stuffff}
\end{align}
$u^{sn/2}\curbnom{-(n+\ell)}{k-n+s}$ is the coefficient of $t^{k-n+s}$ in $u^{(n^2-kn)/2}K_{-(n+\ell)}(tu^{n/2},u)$ and $\curbnom{\ell}{s}$ is the coefficient of $t^{-s}$ in $K_\ell(t^{-1},u)$. Therefore, the sum in \eqref{stuffff} is the coefficient of $t^{k-n}$ in 
\begin{align*}
u^{(n^2-kn)/2}K_{-(n+\ell)}(tu^{n/2},u)K_\ell(t^{-1},u)&=u^{(n^2-kn)/2}t^{-\ell}K_{-(n+\ell)}(tu^{n/2},u)K_\ell(t,u)\\
&=u^{(n^2-kn)/2}t^{-\ell}K_{-n}(tu^{(n+\ell)/2},u)
\end{align*}
which is
\begin{eqnarray*}
u^{\frac{\ell^2+\ell k}{2}}\curbnom{-n}{k-n+\ell}&=&(-1)^{k-n+\ell}u^{\frac{\ell^2+\ell k}{2}}\curbnom{k+\ell-1}{n-1}\\
&=&(-1)^{k-n+\ell}u^{\frac{\ell^2+\ell k-(n-1)(k+\ell-n)}{2}}\sqnom{k+\ell-1}{n-1}\
\end{eqnarray*}
and we get
\[P^n_{k,k+2\ell}=u^{\ell^2+\ell k-n\ell}\frac{[k+2\ell]}{[n+\ell]}\sqnom{n+\ell}{n}\sqnom{k+\ell-1}{n-1}
\]
\end{proof}

%------------------------------------------------------------------------------------------%
\bibliography{biblio}
\bibliographystyle{amsalpha}
%------------------------------------------------------------------------------------------%

\end{document}